\documentclass[11pt]{amsart}
\usepackage{a4wide}
\usepackage{amsmath,amssymb,amsfonts,amsthm,dutchcal}
\usepackage{tikz}
\usetikzlibrary{decorations.pathreplacing}
\usepackage{color}
\usepackage{graphicx}

\usepackage{cancel}

\usepackage[colorlinks=true,urlcolor=blue,
citecolor=red,linkcolor=blue,linktocpage,pdfpagelabels,
bookmarksnumbered,bookmarksopen]{hyperref}

\theoremstyle{plain}

\newtheorem{theorem}{Theorem}[section]
\newtheorem{lemma}[theorem]{Lemma}
\newtheorem{proposition}[theorem]{Proposition}

\newtheorem{remark}[theorem]{Remark}
\newtheorem{definition}[theorem]{Definition}
\theoremstyle{definition}
\numberwithin{equation}{section}

\def\Om{\Omega}
\def\R{\mathbb{R}}
\def\S{\mathbb{S}}

\def\dist{\textup{dist}}

\def\H{\mathcal{H}}

\def\C{\mathcal{C}}
\def\Co{\mathbf{C}}
\def\J{J_{\lambda,\Co}}
\def\Chi#1{\hbox{{\large $\chi$}{\Large $_{_{#1}}$}}}

\newcommand{\res}{\mathop{\hbox{\vrule height 7pt width .5pt depth 0pt
\vrule height .5pt width 6pt depth 0pt}}\nolimits}

\newcommand{\wstar}{\overset{\!\!*}\rightharpoonup}

\newcommand{\e}{\varepsilon}

\newcommand{\pa}{\partial}

\newcommand{\medint}{-\kern -,375cm\int}
\newcommand{\medintinrigo}{-\kern -,315cm\int}
\newcommand{\wto}{\rightharpoonup}

\def\beq{\begin{equation}}
\def\eeq{\end{equation}}

\title[Capillary energy outside convex cylinders]{The isoperimetric inequality for the capillary energy outside convex cylinders}

\author{Nicola Fusco}

\author{Vesa Julin}

\author{Massimiliano Morini}

\begin{document}

\begin{abstract} 
 We study the isoperimetric problem for capillary surfaces with a general contact angle \( \theta \in (0, \pi) \), outside convex infinite cylinders with arbitrary two-dimensional convex section. We prove that the capillary energy of any surface supported on any such convex cylinder is strictly larger than that of a spherical cap with the same volume and the same contact angle on a flat support, unless the surface is itself a spherical cap resting on a facet of the cylinder. In this class of convex sets, our result extends for the first time the well-known Choe-Ghomi-Ritoré relative isoperimetric inequality, corresponding to the case \( \theta = \pi/2 \), to general angles.
\end{abstract}

\maketitle

\section{Introduction}

Let $\Co \subset \R^N$ be a closed convex set with nonempty interior. Given a set of finite perimeter $E\subset\R^N\setminus\Co$ and $\lambda\in(-1,1)$ we define the capillary energy as 
\[
\J(E):=P(E; \R^N \setminus \Co) - \lambda \H^{N-1}(\pa^*E \cap \pa \Co).
\]
Here, for any Borel set $G$, $P(E;G)=\H^{N-1}(\pa E^*\cap G)$  and $\pa^*E$ is the reduced boundary of $E$ (for the definitions and the relevant properties see \cite{AmbrosioFuscoPallara00, Maggi12}).
The  capillary energy has a natural physical motivation as it models a liquid  drop supported  on a given substrate and we refer to \cite{Finnbook} for   a comprehensive introduction to the topic.

For every  $v >0$ we consider the isoperimetric problem 
\begin{equation} 
\label{def:min-prob}
I_{\Co}(v) := \inf\{ \J(E):\,\,E \subset  \R^N \setminus \Co, \, |E| =v \}.
\end{equation}
 When the convex set $\Co$ is bounded the problem \eqref{def:min-prob} has  a minimizer, and if $\Co$ is in addition  smooth, then the minimizer is  smooth up to a small singular set  and the free boundary $\pa E \setminus \Co$ meets the surface $\pa \Co$ with an angle  given  the classical Young's law  \cite{Taylor77, De-PhilippisMaggi15}. We also mention the recent work  related to Allard type regularity for  critical sets  \cite{gasparetto}. When the convex set is unbounded the problem \eqref{def:min-prob} might not admit  a minimizer. This happens for instance  when $\Co=\C\times \R\subset \R^3$, with $\C$ the epigraph of a parabola. In this case, as a consequence of our main Theorem~\ref{thm1}, minimizing sequences  slide upwards to infinity along the boundary of $\Co$ and the isoperimetric profile \eqref{def:min-prob}  agrees with the profile given by the half-space. In the case $\lambda = 0$  the problem for unbounded general convex sets $\Co$  is studied in \cite{FMMN}.

 The issue we want to address  here is to find the convex sets $\Co$ for which the value of \eqref{def:min-prob} is the smallest.  
In the case $\lambda= 0$ the problem reduces to  the relative isoperimetric inequality outside convex sets proven by Choe-Ghomi-Ritor\'e in \cite{ChGhoRi07}:  using the tools developed in  \cite{ChGhoRi06} they show that the half-space gives the lowest value for \eqref{def:min-prob} . On the other hand,  rather surprisingly  the capillary case with general $\lambda \neq 0$ has remained completely open until now as  all the methods devised for the relative isoperimetric problem do not seem to be applicable to  \eqref{def:min-prob}. Here, we solve it in the case of infinite convex cylinders. In fact, our result holds in every dimension  for convex sets, whose normals span  a two-dimensional plane.  

In order to state our main result we denote the half-space by $\textbf{H} = \{ x \in \R^N : x_N < \lambda \}$  and by $B^\lambda_r$  the solid spherical cap 
\[
B^{\lambda}_r  = \{ x \in B_r : x_N > \lambda\}.
\]
Given $v>0$, we let  $B^\lambda[v] = B^\lambda_r$ denote the spherical cap with radius $r$ such that  $|B^\lambda[v]|=v$. Our main result is the following. 
%


\begin{figure}[h!]
\includegraphics[scale=0.25]{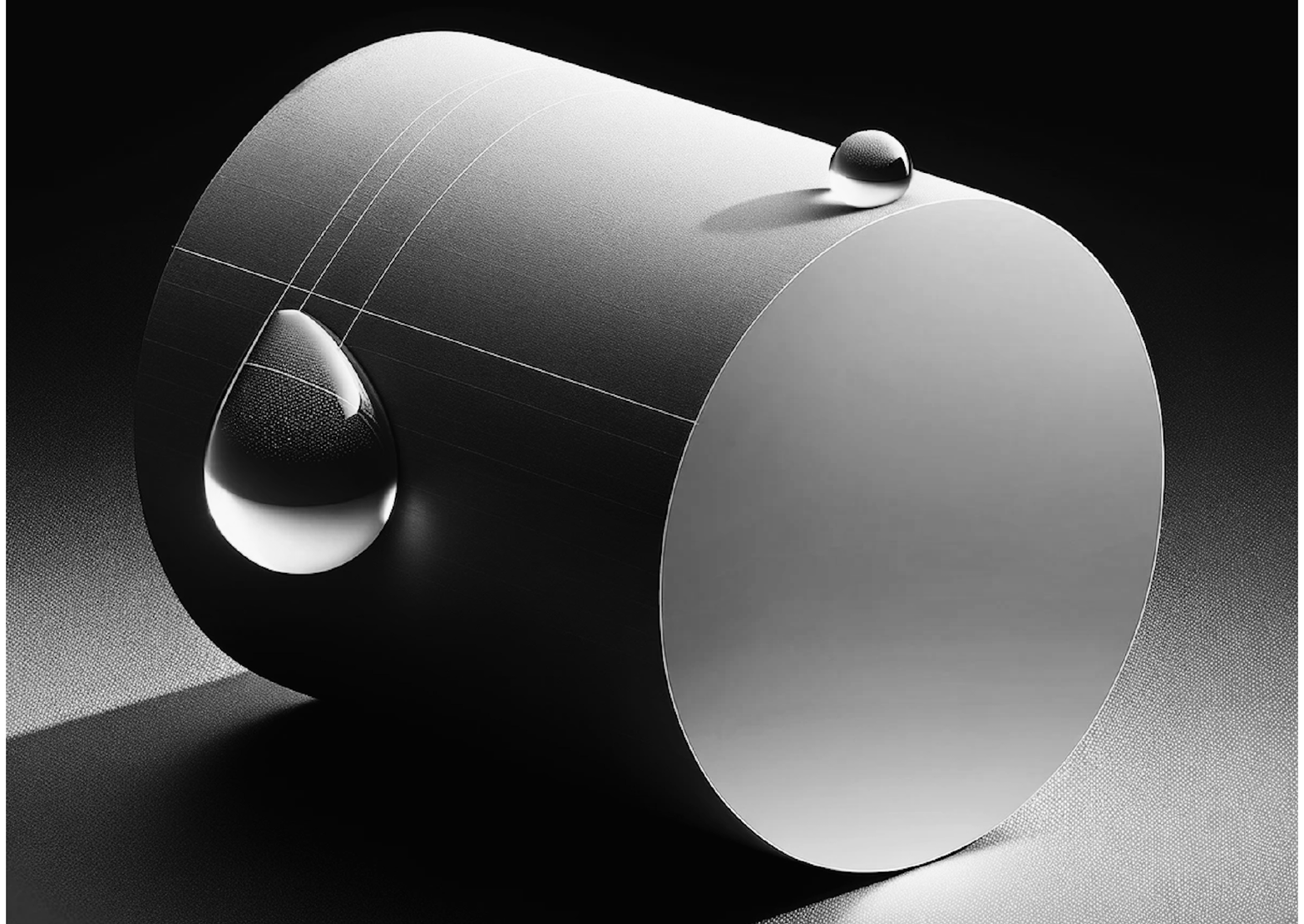}
\caption{Droplets supported on a convex cylinder}
\label{fig-cylinder}
\end{figure}

\begin{theorem}
\label{thm1}
Let  $\lambda \in (-1,1)$ and let $\Co$ be of the form $\C\times\R^{N-2}$, where $\C\subset\R^2$ is a closed convex set  of the plane with nonempty interior. \footnote{For the case of a convex set with empty interior, see Remark~\ref{empty}.} For every set of finite perimeter $E\subset\R^N\setminus\Co$ such that $|E|=v$ we have
\beq\label{main1}
\J(E)\geq J_{\lambda,\textbf{H}}(B^\lambda[v]).   
\eeq
Moreover the equality holds if and only if $E$ sits on a facet of $\Co$ and  is isometric to  $B^\lambda[v]$.
\end{theorem}

Note that in the physical case $N=3$ the above theorem holds for all convex infinite cylinders with arbitrary two-dimensional section. We highlight also that we do not assume any regularity on $\Co$. In particular, the theorem above applies to the case where $\Co$ is an infinite wedge and shows  that the capillary energy of a droplet sitting outside a  wedge and wetting its ridge has energy strictly larger than a spherical cap lying on a flat surface, a fact that, to the best of our knowledge, was not proven before.

Instead, the capillary isoperimetric problem \emph{inside} a convex wedge was studied in \cite{lopez} where it is  proved that  the minimizer of the capillary energy is a spherical cap centered at the ridge of the wedge. The same result holds also for critical points, as a consequence of the generalized  Heinze-Karcher inequality proven in \cite{JWXZ}. Instead, Theorem \ref{thm1} implies that the capillary isoperimetric problem  {\em outside}  a convex wedge has  the opposite behavior, since the minimizer is a spherical cap sitting on either facet of the wedge away from the ridge.

As we already mentioned,  the case $\lambda =0$ of  Theorem~\ref{thm1} the problem \eqref{def:min-prob}  is the relative  isoperimetric  inequality  due to  Choe-Ghomi-Ritor\'e \cite{ChGhoRi07}, see also   \cite{FM2023} for  the rigidity, i.e., the characterization of the  equality case for general, possibly nonsmooth convex sets. We also refer to
\cite{LiuWangWeng}  for an alternative proof of the same  inequality and
 to \cite{krummel} for the problem in higher  codimension.

In order to prove  Theorem \ref{thm1} we need to introduce some novel methods, that will be explained in more details in Section~\ref{sec:overview} below. Indeed, the approach  based on normal cones developed in  \cite{ChGhoRi06, ChGhoRi07} for the case $\lambda =0$ (and further refined in  \cite{FM2023})  gives only information on the free boundary $\pa E \setminus \Co$, while the contact region $\pa E \cap \pa \Co$ remains invisible. In order to overcome this,  we adapt  to  the capillary problem the  ABP-method, originally introduced by  Cabr\'e  for the standard isoperimetric inequality. This approach already appears in   \cite{LiuWangWeng}  for $\lambda = 0$ and here we develop it in the case $\lambda\not=0$.  However, in order to apply the  ABP-method one needs a subtle estimate on the set of subdifferentials of the function $u$ solving problem \eqref{eq:neumann-0} below, as we will explain in the next subsection. This difficulty appears already in the case $\lambda=0$ and the authors in \cite{LiuWangWeng} overcome it by rewriting  the aforementioned estimate in terms of suitable restricted normal cones to the graph of $u$ and by using some of the results proved in \cite{ChGhoRi06}. This argument does not seem to generalize to the case $\lambda \neq 0$. 

Instead of relying on normal cones, we first prove that  the solution  to the Neumann problem  \eqref{eq:neumann-0}  is a viscosity supersolution of the same problem and exploit this property  to  study directly the geometry of  its subdifferentials. 
Since the subdifferentials of a function are convex and their union is the whole space, the new point of view combined with a  discretization procedure  leads us  to reformulate the aforementioned estimate in terms of an estimate on finite convex partitions of the space made up of  subdifferentials of discrete  functions. This seems to be a rather complicated combinatorial problem that for $\lambda\not=0$ we can solve only in the planar case, allowing us to treat  convex cylinders of the form $\C\times\R^{N-2}$. 
We note however that  this method gives an easy proof of the inequality \eqref{eq:crucial-est} and thus of \eqref{def:min-prob} for all convex sets $\Co$ when $\lambda = 0$.
 Finally, we conjecture that \eqref{main1} holds for all convex sets.

We next give an overview of the ABP-argument which we use in the  proof of Theorem \ref{thm1}.
  
\subsection{Overview of the proof}\label{sec:overview}

The proof of Theorem \ref{thm1} is based on the ABP-method applied to Neumann problem \eqref{eq:neumann-0}. We note that in the context of isoperimetric problems this method was used first time by  Cabr\'e in  \cite{cabre2000, cabre2008} and further generalized in \cite{CRS2016}. In this paper we extend this method to the  capillary isoperimetric problem. Let us sketch the proof and outline the main challenges of the argument. 

By scaling we may reduce to the case $|E| = |B^\lambda|$.  Assume for simplicity that the set $E$ is regular in which case we denote it by $E=\Omega$. To be more precise we assume  that $\Omega \subset \R^N\setminus\Co$ is a Lipschitz regular open set with $|\Omega| = |B^\lambda|$, such that $\Sigma = \pa \Omega  \setminus\Co$   and $\Gamma = \pa \Omega \cap \Co$ are smooth embedded manifolds with a common boundary  denoted  by $\gamma$. Let $u :\bar \Omega \to \R$ be the solution of the Neumann boundary problem  
\begin{equation}\label{eq:neumann-0} 
\begin{cases}
&\Delta u = c \quad \text{in } \, \Omega \\
&\pa_\nu u = 1 \quad \text{on } \, \Sigma\\ 
&\pa_\nu u = -\lambda  \quad \text{on } \, \Gamma,  
\end{cases}
\end{equation}
where $\lambda\in(-1,1)$ and the constant 
\begin{equation}\label{eq:compatibility} 
c = \frac{\H^{N-1}(\Sigma) - \lambda \H^{N-1}(\Gamma)}{|\Omega|} = \frac{\J(\Omega )}{|\Omega|}
\end{equation}
is the one prescribed by the compatibility condition. We point out that in the case  $\Omega = B^\lambda$  up to an additive constant  $u(x) = \frac12 |x|^2$ and $c = N$.  

Let us denote the convex envelope of $u$ by $\hat u$, i.e., the largest convex function that is below $u$. Note that if $x \in \Omega$ is such that $u(x) = \hat u(x)$, then the \emph{subdifferential} of $u$ at $x$ is non-empty
\[
J_{\bar \Omega}u(x) = \{ \xi \in \R^N : u(y) -u(x) \geq \xi \cdot (y-x) \quad \text{for all } \, y \in \bar \Omega\} \neq \emptyset 
\]
and since $u$ is smooth in $\Omega$ it holds $\xi = \nabla u(x)$ and $\nabla^2 u(x) \geq 0$.  We denote the set of points for which $J_{\bar \Omega}u(x) $ is non-empty by $\Omega^+$ and the  union of  subdifferentials by $\mathcal A_u$ 
\[
\mathcal A_u = \bigcup_{x \in \Omega}J_{\bar \Omega}u(x). 
\]
In order to carry one the ABP-argument one needs to show the following crucial estimate 
\begin{equation}\label{eq:crucial-est} 
|\mathcal A_u| \geq |B^\lambda| 
\end{equation}
by somehow exploiting the Neumann boundary conditions in \eqref{eq:neumann-0} and the geometry of $\Co$.

Once \eqref{eq:crucial-est} is proven, the claim then follows  by using  the Area Formula, the arithmetic-geometric mean inequality and recalling the value of $c$ in \eqref{eq:compatibility} 
\[
\begin{split}
|\Omega|=|B^\lambda|\leq |\mathcal A_u| =  |\nabla u(\Om^+)| \leq \int_{\Omega^+}\det\nabla^2u \, dx \leq \int_{\Omega^+}\frac{(\Delta u)^N}{N^N}\, dx \leq \left( \frac{\J(\Omega )}{|\Omega| N}\right)^N |\Omega|. 
\end{split}
\]
The above chain of inequalities gives the conclusion as 
\[
\J(B^\lambda) = N |B^\lambda| = N|\Omega|. 
\]
The case of a general  set of finite perimeter $E$ follows by  an approximation argument. In fact, a more refined approximation argument allows us also  to characterize the case of equality, see Lemma~\ref{appro}.

It is then clear that the most relevant estimate is \eqref{eq:crucial-est} which turns out to be challenging to prove. Indeed, as observed in \cite{LiuWangWeng}  the inclusion  $B^\lambda \subset  \mathcal A_u$ does not hold  in general and we need to develop a more subtle argument to overcome  the problem. The same estimate was already  proven in \cite{LiuWangWeng}  in the case  $\lambda = 0$ by reformulating the problem in terms of suitable restricted normal cones to the graph of $u$ in the spirit of  \cite{ChGhoRi06}. However their  argument does not generalize to the case $\lambda \neq 0$. 

In order to deal with  the case of  general $\lambda$'s, we   develop the following novel argument. We first show that the (variational) solution to \eqref{eq:neumann-0}  is a viscosity supersolution  to the same problem, in the sense of Definition~\ref{def:visco}. Note that the latter definition is stronger than the one given in  \cite{dupuisishii} and therefore the supersolution property in the above sense does not follow from known results. Using this property we are able to relate the subdifferentials of $u$ in $\Omega$ with the subdifferentials of the restriction of $u$ to $\Gamma$  proving the following inclusion, see Lemma~\ref{lem:u-gamma},
\[
\mathcal B_{u_{\Gamma}}^\lambda \cap B_1 \subset\mathcal A_u \cap B_1,
\]
where $\mathcal B_{u_{\Gamma}}^\lambda=\bigcup_{x\in \Gamma} \{ \xi \in J_{\Gamma} u(x)  : \,\xi\cdot\nu_\Co(x)> \lambda \}$, $ J_{\Gamma} u(x)$ is the subdifferential at $x$ of the restriction of $u$ to $\Gamma$ and $\nu_\Co(x)$ stands for the outer normal to $\Co$ at $x$. This inclusion leads to the proof of \eqref{eq:crucial-est}  provided we are able to show that
\beq\label{introcr}
\qquad\qquad\qquad\qquad |\mathcal B_{v}^\lambda\cap B_1|\geq|B^\lambda| \qquad \text{for any $v:K\to\R$, with $K\subset\pa\Co$ compact.}
\eeq
 In fact,  it is enough to prove \eqref{introcr} only for discrete sets $K\subset\pa\Co$, see Lemma~\ref{lem:ABP}. Note that property \eqref{introcr} has nothing to do  with Neumann problem \eqref{eq:neumann-0} and it only depends on the geometry of  $\Co$. We call this property \emph{$\lambda$-ABP property}, see Definition \ref{def:ABP}.  It can be easily shown that any convex sets satisfies the $0$-ABP property, see  Proposition~\ref{prop:0ABP}.  As already mentioned, in this way we obtain a new simple proof of \eqref{eq:crucial-est} (and thus of the Choe-Ghomi-Ritor\'e  isoperimetric inequality) in the case $\lambda=0$. 
 
 The case $\lambda\not=0$ is considerably more difficult. In fact, proving the $\lambda$-APB property is equivalent to show that, given any convex partition $A_1,\dots,A_k$ of $\R^N$, with $A_i$ the subdifferential at $x_i$ of a function $v:\{x_1,\dots,x_k\}\subset\pa\Co\to\R$, then
\beq\label{introcr1}
 \sum_{i=1}^k\big|A_i\cap\{\xi\in B_1:\,\,\xi\cdot\nu_i>\lambda\}\big|\geq|B^\lambda|,
 \eeq
 where $\nu_i$ is the exterior normal to $\Co$ at $x_i$. The proof of the above inequality is a difficult combinatorial problem, and we are able to show it  only  when all the normals to $\Co$ lie in a 2-dimensional plane, that is in the case $\Co=\mathcal C\times\R^{N-2}$. Indeed, in this case, by a slicing argument we are able to reduce the proof of \eqref{introcr1} to a similar estimate for convex partitions in the plane. The proof of the latter, which is the content   of Section~\ref{sec:3.1}, is still very complicated and we achieve it by  studying the left hand side of \eqref{introcr1} as a function of $\lambda$ by means of geometrical and analytical arguments.

\section{The ABP approach for the capillary energy}

In this section we set up the tools that we need for the  ABP argument.  To this aim, in Section~\ref{sec:neu}  we establish  the crucial viscosity supersolution property for  the variational solutions of the Neumann problem \eqref{eq:neumann-0}. In  Section~\ref{sec:subdif} we introduce a useful notion of restricted subdifferential which enables us to  reduces the inequality \eqref{eq:crucial-est}  to a property of the convex set $\Co$. As we already mentioned, we call it  $\lambda$-ABP property and give its definition in Section~\ref{sec:lambdaABP}  (Definition \ref{def:ABP}). At the end of the section we establish the relative isoperimetric inequality for the $\lambda$-capillary functional outside the convex sets that satisfy such a property.

As we mentioned in the introduction, we will first prove the relative isoperimetric inequality for regular sets. In order to emphasize this  we always denote $E = \Omega$ when the set is assumed to be regular, i.e., it satisfies 
\beq\label{Coreg}
 \text{$\Co\subset\R^N$ is a  closed convex set  of class $C^{2}$} 
 \eeq
 and  \beq\label{Omreg} 
\begin{split}
&\text{$\Om\subset\R^N\setminus\Co$ is  a bounded Lipschitz set such that  $\Sigma:=\pa\Om\setminus\Co$ } \\
&\qquad\qquad\qquad \text {is a $(N-1)$-manifold with boundary of class $C^{2}$}.
\end{split}
\eeq 
We call the boundary $\Sigma$ the \emph{free interface} and denote $\Gamma:=\pa\Om\cap \Co$, which we call the \emph{wetted region} which is also an embedded  $C^2$-regular $(N-1)$-manifold with boundary. Note that  $\Gamma$ and $\Sigma$ share the same boundary, which we denote by $\gamma:=\overline \Sigma\cap \Co$  and which  by the assumption is a $(N-2)$-manifold of class $C^{2}$. We call  $\gamma$  the \emph{contact set} of $\Sigma$ with $\Co$. Moreover, we will throughout the section  assume  $|\Omega| = |B^\lambda|$ if not otherwise mentioned. 

We will denote by $\nu_\Om$ and $\nu_\Co$ the outer unit normal to $\pa \Om$ and to $\pa \Co$, respectively. We also denote by  $\nu_\Sigma=\nu_\Om$ the outer unit normal field on  $\Sigma$, which by our assumptions admits a continuous  extension at $\gamma$, still denoted by $\nu_\Sigma$. We also set  $\nu_\Gamma=-\nu_\Co$ on $\Gamma$. Note that the Lipschitz regularity of $\Om$ yields 
 \[
 \nu_\Gamma \cdot \nu_\Sigma> -1\qquad\text{on }\gamma.
 \]
Finally, we define the $\e$-neighborhood of a generic set $F \subset \R^N$ as $B_\e + F = \{ x \in \R^N : \text{dist}(x,F)< \e\}$. 

\subsection{The Neumann problem}\label{sec:neu}

In this section we consider the Neumann problem under the assumptions  \eqref{Coreg} and \eqref{Omreg}. We may also assume without loss of generality $\Omega$ to be connected. Note that even in this case $\Omega$  is still merely a Lipschitz domain and therefore the high regularity of $u$ up to boundary is not granted.  However, it turns out that we only need the solution to attain the boundary values in the viscosity sense, for which H\"older continuity up to the boundary is enough. To this aim we consider the  variational solution of the problem \eqref{eq:neumann-0} which by definition is  a function $u\in H^1(\Om)$ such that  it holds 
\beq\label{eq:general}
\int_{\Omega} \nabla u \cdot \nabla \varphi\,dx= -\int_\Om c \varphi\,dx+\int_{\pa\Om}g\varphi\,d\H^{N-1}  
\eeq
for all $\varphi\in H^1(\Omega)$, where $c$ is given by \eqref{eq:compatibility},  
\beq\label{g}
g \equiv 1\text{ on  } \Sigma \quad\text{and}\quad g \equiv -\lambda \quad\text{on } \Gamma \setminus \gamma.
\eeq
 Since $\Omega$ is bounded and Lipschitz regular,  $g$ is bounded and we have the  following compatibility condition 
\[
c |\Omega |   =-\int_{\pa\Omega}g\,d\H^{N-1},
\] 
there exists a unique  (up to an additive constant)  variational solution of \eqref{eq:neumann-0}. Moreover,  by standard elliptic regularity theory the variational solution is H\"older continuous up to the boundary,  see for instance \cite{Nittka}.

Let us proceed to the notion of viscosity solution. In fact, it turns out that we need only the concept of viscosity supersolution for the ABP-argument and therefore we reduce only to that. Here is the definition we need. 
\begin{definition}
\label{def:visco}
A lower semicontinuous function $u : \overline \Omega \to \R$ is a \emph{viscosity supersolution}  of \eqref{eq:neumann-0} if  whenever $u-\varphi$ has a local minimum at $x_0\in \overline\Om$ for  $\varphi \in C^2(\R^N)$,  then
\beq\label{eq:visco}
\begin{cases}
-\Delta \varphi(x_0) \geq -c & \text{if } \, x_0 \in \Omega,\\
\partial_{\nu_\Sigma}\varphi(x_0)  -1 \geq 1 & \text{if } \, x_0 \in \Sigma \setminus \gamma,\\
\partial_{\nu_\Gamma} \varphi(x_0)+\lambda \geq 0 & \text{if } \, x_0 \in \Gamma\setminus \gamma,\\
 \max\{\partial_{\nu_\Sigma}\varphi(x_0)  -1, \partial_{\nu_\Gamma}\varphi(x_0) \cdot \nu_\Gamma(x_0)+\lambda \}\geq 0 & \text{if }x_0\in \gamma,
\end{cases}
\eeq
where $\partial_{\nu_\Sigma}\varphi(x_0) = \nabla \varphi(x_0) \cdot \nu_{\Sigma}(x_0)$ and $\partial_{\nu_\Gamma}\varphi(x_0) = \nabla \varphi(x_0) \cdot \nu_{\Gamma}(x_0)$. 
\end{definition}

We can now prove the main result of the section.
\begin{proposition}\label{th:corner}
Let $\Om$, $\Co$ be as in \eqref{Coreg} and \eqref{Omreg}. Then the variational solution of \eqref{eq:neumann-0} is a viscosity supersolution in the sense of Definition~\ref{def:visco}.
\end{proposition}
\begin{proof} 

 Since the equation is satisfied classically in $\Om$ and the Neumann boundary conditions are achieved in a classical sense at $\pa\Om\setminus \gamma$, it will be enough to check the property on $\gamma$. To this aim, assume that  $\varphi\in C^2(\R^N)$ and $x_0\in \gamma$ are such that $(u-\varphi)(x)\geq 0$, with equality achieved only at $x_0$. We start by showing that 
\beq\label{cond1}
 \max\{-\Delta \varphi(x_0)+c, \pa_{\nu_\Sigma}\varphi(x_0)-1,\pa_{\nu_\Gamma}\varphi(x_0)+\lambda\}\geq 0.
\eeq
We argue by contradiction, assuming that 
\[
\max\{-\Delta \varphi(x_0)+c, \pa_{\nu_\Sigma}\varphi(x_0)-1,\pa_{\nu_\Gamma}\varphi(x_0)+\lambda\}< 0.
\]
By continuity we may find a small ball $B_r(x_0)$ such that 
\beq\label{cont2}
-\Delta \varphi+c<0 \quad \text{in }B_r(x_0)\qquad\text{and}\qquad \pa_{\nu}\varphi-g < 0\quad\text{ on } (\pa\Omega\cap B_r(x_0))\setminus \gamma\,,
\eeq
with $g$ defined in \eqref{g}.
Then, setting $w:=u-\varphi$,  $h:=-c+\Delta \varphi$, $f=\pa_{\nu}(u-\varphi)=g-\pa_{\nu}\varphi$, we have that $w$ is a variational solution of
\[
\begin{cases}
-\Delta w= h & \text{in }\Om \cap B_r(x_0),\\
\pa_\nu w=f & \text{ in } \pa\Omega \cap B_r(x_0)
\end{cases}
\] 
that is, 
\beq\label{debole}
\int_{\Omega}\nabla w\cdot \nabla\psi\, dx=\int_{\Om}h\psi\, dx+\int_{\pa \Om} f\psi\, d\H^{N-1}
\eeq
for all $\psi\in H^1(\Om)$ with $\psi =0$ in $\Om\setminus B_{r}(x_0)$.  Let us now choose $\psi:=\min\{w-\e, 0\}$ and note that for $\e>0$ small enough 
\[
\psi=\min\{w-\e, 0\}=0\quad\text{in }\Om\setminus B_{r}(x_0)\,.
\]
 Then, \eqref{debole} combined with the fact that $h>0$ in $\Om\cap B_r(x_0)$ and $f>0$ in $\pa \Om\cap B_r(x_0)$, yield
\[
\int_{\Omega}\big|\nabla \big(\min\{w-\e, 0\}\big)\big|^2 dx= \int_{\Om\cap B_r(x_0)}h\psi\, dx+\int_{\pa \Om\cap B_r(x_0)} f\psi\, d\H^{N-1}\leq 0\,
\]
and in turn $\min\{w-\e, 0\}=0$ in $\Om$. This is impossible since $w-\e<0$ in a neighborhood of $x_0$. Thus \eqref{cond1} is established.

The inequality \eqref{cond1}  is not good enough, since we only want information on the boundary.  We thus claim that in fact
\beq\label{cond3}
 \max\{ \pa_{\nu_\Sigma}\varphi(x_0)-1,\pa_{\nu_\Gamma}\varphi(x_0)+\lambda\}\geq 0. 
\eeq
To this aim we observe that for any $x_0\in\gamma$ there exists a ball $B_r(\bar x)\subset\R^N\setminus\overline\Om$ with $x_0\in\pa B_r(\bar x)$ (it is enough to take a  ball contained in $\Co$ and tangent to $x_0$, which is possible by the $C^{2}$ assumption on $\Co$).
 By translating and dilating we may assume for simplicity that $\bar x=0$ and that $r=1$.   We perturb the test function $\varphi$ by a functions $\psi_q\in C^2(\R^n \setminus \{ 0 \} )$ that we define as 
\[
\psi_q(x) = \frac{1}{q^{3/2}}(|x|^{-q} - 1)
\]
where $q>0$ is a large number to be chosen. Then by the exterior ball condition we have that $\psi_q(x) \leq 0$ for $x \in \Omega$ while since $x_0 \in \pa B_1$ it holds  $\psi_q(x_0)= 0$.  Moreover by a direct computation we see that  
\[
\Delta \psi_q(x_0) \geq \frac{\sqrt{q}}{2} \quad \text{and} \quad |\nabla \psi_q(x_0)|= \frac{1}{\sqrt{q}},
\]
for $q$ sufficiently large.
We define a new test function 
\[
\varphi_q(x) = \varphi(x) + \psi_q(x).
\]
By construction it  holds $\varphi_q \leq\varphi$ in $\Omega $ and $\varphi_q(x_0) = \varphi(x_0)$, hence  $x_0$ is still a local minimum for $u-\varphi_q$. Thus
\[
-\Delta \varphi_q(x_0)+c\leq -\frac{\sqrt{q}}{2}-\Delta \varphi(x_0)<0
\]
when $q$ is large. Therefore, for $q$ large, from \eqref{cond1} we obtain \eqref{cond3} with $\varphi$ replaced by $\varphi_q$. Finally, letting $q\to\infty$, since $\nabla\varphi_q(x_0)\to\nabla\varphi(x_0)$ we obtain \eqref{cond3}. This concludes the proof.
\end{proof}

\subsection{Subdifferentials}\label{sec:subdif}
We need some notation in order  to proceed. Given  $X\subset\R^N$, a function $u : X \to \R$, a subset $Y\subset X$  and a  point $x\in Y$ we define  the subdifferential $J_Yu(x)$ as
\[
J_Yu(x):=\big\{\xi\in\R^N:\,u(y) - u(x) \geq \xi \cdot (y-x) \,\,\,  \text{for all} \, y \in Y\big\}.
\]
We  note that   $Y$ may even be a discrete set.  

\begin{remark}\label{rm:ovv}
Note that if  $Y$ is bounded and $u$ is bounded from below in $Y$, then
\[
\bigcup_{x\in Y}J_Y u(x)=\R^N.
\]
Indeed, for any $\xi\in \R^N$, we may find $c<0$ so negative that $\displaystyle\sup_{x\in Y}\big(-c+\xi\cdot x-u(x))<0$. Setting 
\[
t_0:=\sup\{ t>0:\,    -c+\xi\cdot x+s<u(x)\text{ for all $x\in Y$ and for $s\in(0,t)$}\},
\]
we clearly have  $-c+\xi\cdot \bar x+t_0=u(\bar x)$ for some $\bar x\in Y$  and  $-c+\xi\cdot  x+t_0\leq u(x)$ for all $x\in Y$; that is, $\xi\in J_Y u(\bar x)$.
\end{remark}
Let $\Om$, $\Gamma$ and $\Sigma$  be as in \eqref{Coreg} and \eqref{Omreg}. Recall that the crucial inequality \eqref{eq:crucial-est}  involves the set $\mathcal A_u $ which we define for any given  function $u:\overline\Om\to\R$ as
\beq\label{alambdau}
\mathcal A_u := \bigcup_{x\in\Omega} \{\xi \in \R^{N} : \xi \in J_{\overline\Om}\,u(x) \}.
\eeq
If the function $u$ in  \eqref{alambdau} is the  solution of \eqref{eq:neumann-0}, then it turns out that all relevant information is contained in its restriction to $\Gamma$. This leads us to define the following union of subdifferentials for functions defined on the boundary of $\Co$,     $v : K \subset \partial \Co \to \R$ 
\beq\label{blambdau}
\mathcal B_{v}^\lambda :=\bigcup_{x\in K} \{ \xi \in J_K v(x)  : \,\xi\cdot\nu_\Co(x)> \lambda \},
\eeq
where $\lambda \in (-1,1)$.  If $\Om$ and $u:\overline\Om\to\R$ are as above then we denote the restriction of $u$ to $\Gamma$  by $u_\Gamma$ and  by the previous notation we have 
\beq\label{blambdau2}
\mathcal B_{u_{\Gamma}}^\lambda =\bigcup_{x\in \Gamma} \{ \xi \in J_{\Gamma} u(x)  : \,\xi\cdot\nu_\Co(x)> \lambda \}.
\eeq
We transform the information of the Neumann boundary problem \eqref{eq:neumann-0}  into the following information on the set $\mathcal B_{u_{\Gamma}}^\lambda$. 
\begin{lemma}
\label{lem:u-gamma}
Let $\Om$,  $\Sigma$ and $\Co$ be as in \eqref{Coreg} and \eqref{Omreg}. Let  $u$ be the variational solution of \eqref{eq:neumann-0}. Then it holds 
\[
\mathcal B_{u_{\Gamma}}^\lambda \cap B_1 \subset\mathcal A_u \cap B_1,
\]
where $\mathcal B_{u_{\Gamma}}^\lambda$ is given by \eqref{blambdau2}.
\end{lemma} 

\begin{proof}
Recall that $u$ is continuous up to the boundary and it is a viscosity supersolution of \eqref{eq:neumann-0} by Proposition~\ref{th:corner}.
Fix $\xi \in \mathcal B^\lambda_{u_\Gamma} \cap B_1$. Recall that this means that $\xi\in J_\Gamma u(x)$, for some $x\in\Gamma$ and that $\xi\cdot\nu_\Co(x)>\lambda$ (and of course $|\xi|<1)$. We need to show that  $\xi \in \mathcal A_u$, i.e., $\xi \in J_{\overline\Om}u(\bar x)$ for some $\bar x \in \Om$.

 First we claim  that $\xi\not\in J_{\overline\Om}u(x)$, where $x \in \Gamma$ is the point for which $\xi\in J_\Gamma u(x)$. Indeed, assume the opposite, that is
\[
u(y)\geq u(x) + \xi \cdot (y- x)=: \varphi(y) \qquad \text{for all } \, y \in \overline \Omega.
\]
In particular,  $x$ is the minimum point of $u - \varphi$ and since $u$ is a viscosity supersolution and $\nabla \varphi(x) = \xi$, see Definition~\ref{def:visco}, we have
\[
\begin{cases}
 \xi \cdot \nu_\Gamma( x)+\lambda \geq 0 & \text{if } \,  x\in \Gamma\setminus \gamma,\\
 \max\{ \xi \cdot \nu_\Sigma( x) -1,  \xi \cdot \nu_\Gamma( x)+\lambda \}\geq 0 & \text{if } x\in \gamma.
\end{cases}
\]
From the above condition, since $|\xi|<1$ we have that $ \xi \cdot \nu_\Gamma( x)+\lambda\geq0$, which is impossible as 
$\xi\cdot\nu_\Gamma( x)=-\xi\cdot\nu_\Co( x)<-\lambda$. Therefore $\xi\not\in J_{\overline\Om}u(x)$. 

By the above it holds that  the inequality $u(y)\geq \varphi(y)$ is not true for all $y\in\overline\Omega$. This means that $\bar c=\min_{y\in\overline\Omega}\,(u(y)-\varphi(y))<0$. In turn we have that the hyperplane $s=\varphi(y)+\bar c$ touches from below the graph of $u$ at some point $\bar x\in\overline\Om$. Clearly $\bar x\not\in\Gamma$, since $\varphi(y)+\bar c<u(y)$ for all $y\in\Gamma$. On the other hand, if $\bar x\in \Sigma$, again by the supersolution property, we would have that $\xi\cdot\nu_\Sigma(\bar x)\geq1$, which is impossible since $|\xi|<1$. Therefore $\bar x\in\Omega$, which implies $\xi\in\mathcal A_u$.
\end{proof}

We thus deduce from Lemma \ref{lem:u-gamma} that in order to show the inequality \eqref{eq:crucial-est}, it is enough to  study the  restriction of $u$ on $\Gamma$. It turns  out that,  in terms of the  inequality \eqref{eq:crucial-est}, it is not relevant that   $u_\Gamma$ is  a restriction of the solution of the Neumann problem \eqref{eq:neumann-0}. Indeed, from now on we study  generic functions $v: K \to \R$ which are defined on $K \subset \partial \Co$. The following lemma provides an important property on the structure of the subdifferentials of such functions, which is due to  the convexity of $\Co$. 
\begin{lemma}
\label{lem:v-gamma}
Let $\Co\subset\R^N$ be a closed convex set of class $C^1$, $K \subset \partial \Co$ and $v : K \to \R$.  
If $\xi\in J_K v(x)$ for some $x\in K$, then 
\[
\xi+t\nu_\Co(x)\in J_K v(x) \qquad \text{ for all } \, t>0.
\]
\end{lemma}

\begin{proof}
 If $\xi\in J_K v(x)$, then
\[
v(y) - v(x) \geq \xi \cdot (y-x) \qquad \text{for all } \, y \in K.
\]
By convexity it holds $\nu_\Co(x) \cdot (y - x) \leq 0$ for all $y \in K$. Therefore for any $t>0 $ it holds 
\[
v(y) - v(x) \geq (\xi + t  \nu_\Co(x) ) \cdot (y-x) \qquad \text{for all } \, y \in K,
\]
 that is $\xi + t  \nu_\Co(x) \in J_K v(x)$.   
\end{proof}

\subsection{The $\lambda$-ABP property and the  isoperimetric inequality}\label{sec:lambdaABP}

By Lemma \ref{lem:u-gamma} it is clear  that if we would have $|\mathcal B_{u_{\Gamma}}^\lambda \cap B_1|\geq |B^\lambda|$, where $u_\Gamma$ is the restriction of $u $ on $\Gamma$,  then we would have the inequality \eqref{eq:crucial-est}. As we discussed in the previous section, this property is related to generic functions $v: K \to \R$ ,  where $K \subset \partial \Co$. This in turn means that such a property is only related to the convex set $\Co$ itself. For simplicity we  restrict to functions defined on finite sets.  

\begin{definition}\label{def:ABP}
Let $\Co\subset\R^N$ be a closed convex set of class $C^1$ and $\lambda\in(-1,1)$. We say that $\Co$ has the \emph{$\lambda$-ABP property} if for any finite subset $K\subset\pa\Co$ and for every  $v:K\to\R$
\[
|\mathcal B^\lambda_{v} \cap B_1|\geq|B^\lambda|,
\]
where we recall $B^\lambda=\{x\in B_1:\,x\cdot e_N>\lambda\}$ and $\mathcal B^\lambda_{v}$ is defined in \eqref{blambdau}.
\end{definition}

Let us make a few remarks. First, it is not a priori clear why a convex set would satisfy Definition \ref{def:ABP}, but we show at the end of the section that  every convex set satisfy  $\lambda$-ABP property for $\lambda = 0$. This follows rather easily from Lemma \ref{lem:v-gamma}. In Section 3 we show that convex cylinders have the   $\lambda$-ABP  for every $\lambda\in(-1,1)$. This is much more involved, but the statement is again based on  Lemma  \ref{lem:v-gamma}, which is the only property of the subdifferentials that we need.

We also remark that it follows immediately from the definition that for every $x \in K$ the subdifferential $J_Kv(x)$ is a closed convex set in $\R^N$. Moreover if $K\subset\pa\Co$ is a finite set then there are only finitely many subdifferentials $J_Kv(x)$ and it is also immediate that they are essentially disjoint, i.e., they have disjoint interiors. Therefore by Remark  \ref{rm:ovv}, the subdifferentials $J_Kv(x)$ for $x \in K$ form a convex partition of the space $\R^N$, i.e., partition by convex sets. Hence, the property in Definition \ref{def:ABP} is really a property related to convex partition of $\R^N$ by sets $J_Kv(x)$  which satisfy the statement of Lemma  \ref{lem:v-gamma}. 

We proceed with the following remark.
\begin{remark}\label{rm:ABPscaled}
A simple scaling argument shows that  $\Co$ has the $\lambda$-ABP property if and only if $\eta\Co$ has $\lambda$-ABP property, for any $\eta>0$. Moreover, if $\Co$ has the $\lambda$-ABP property 
then it also holds 
\[
|\mathcal B^\lambda_{v} \cap B_r|\geq|B_r^\lambda|,
\]
for all $r \in (0,1)$.  
\end{remark}
We start by showing  that the above ABP property stated on discrete sets is inherited by all compact subsets of $\pa \Co$, provided that $\pa \Co$ is smooth ($C^1$-regular is in fact enough).
\begin{lemma}\label{lem:ABP}
Let $\Co\subset\R^N$ be a closed convex set  with $C^1$ boundary.  Assume also that $\Co$ satisfies the $\lambda$-ABP property, according to Definition~\ref{def:ABP}, for some $\lambda\in (-1, 1)$. Then, for all compact subsets $K \subset \pa \Co$ and   for  all bounded functions $v : K \to \R$ it holds 
\beq\label{eq:bella0}
|\mathcal B^{\lambda}_{v} \cap B_1|\geq|B^\lambda|,
\eeq
where $\mathcal B^\lambda_{v}$ is  defined in \eqref{blambdau}.
\end{lemma}
\begin{proof}
 Fix $K$, which is  a compact subset of $\pa\Co$ and $v : K \to \R$ which is bounded.  Note first that the convex envelope of $v$ shares the same subdifferentials as $v$ itself. Therefore we may assume that  $v$ is a restriction of a convex function, in particular, $v$ is continuous and at every point  in $K$ its subdifferential is non-empty. Let us choose points $x_1, x_2, x_3, \dots$ in $K$ which are dense in $K$. We define    $K_n= \{ x_1, \dots, x_n\}$ and  the function $v_n : K_n \to \R$ as the restriction of $v$ on $K_n$, i.e., $v_n(x_i)= v(x_i)$ for $i =1,2, \dots, n$.  Note that then 
\beq \label{eq:convergence}
\sup_{x \in K}\inf_{y \in K_n}|v(x) - v_n(y)| \to 0  \quad \text{ as  } \,  n \to \infty. 
\eeq

Let us fix $\lambda' >\lambda$ close to $\lambda$ and show first that  when $n$ is large it holds
\beq\label{eq:bella2}
\mathcal B^{\lambda'}_{v_n}\cap B_1 \subset  \mathcal B^{\lambda}_{v} \cap B_1.
\eeq
To this aim we fix $\xi  \in \mathcal B^{\lambda'}_{v_n}\cap B_1$. This means that  $\xi \in J_{K_n}v_n(x_n)$ for some $x_n \in K_n$ and $\xi \cdot \nu_{\Co}(x_n)> \lambda' $.   By Remark \ref{rm:ovv} there is $x \in K$ such that $\xi \in  J_{K}v(x)$. By the convergence \eqref{eq:convergence} for any $\e>0$ it holds $|x-x_n| < \e$ when $n$ is large. Therefore we have 
\[
\xi \cdot \nu_{\Co}(x)  \geq \xi \cdot \nu_{\Co}(x_n) - |\nu_{\Co}(x)  - \nu_{\Co}(x_n) | > \lambda'   - O(\e) > \lambda 
\] 
when $\e$ is small and $n$ is large. This yields \eqref{eq:bella2}. 

 By Definition~\ref{def:ABP} we have the estimate 
\beq\label{eq:bella}
|\mathcal B^\lambda_{v_n} \cap B_1|\geq|B^{\lambda}|. 
\eeq  
Let us then prove that there exists a universal constant $C>1$ such that 
\beq\label{eq:brutta}
|\mathcal B^{\lambda'}_{v_n} \cap B_1| \geq |\mathcal B^\lambda_{v_n} \cap B_1| - C|\lambda' - \lambda|.
\eeq
We define for $t \geq 0$ 
\[
A_t  :=\bigcup_{x\in K_n} \{ \xi \in J_{K_n} v_n(x)  : \,\xi\cdot\nu_\Co(x)= \lambda+t \}
\]
and recall that by \eqref{blambdau} $\mathcal B_{v_n}^{\lambda+t} =\bigcup_{x\in K_n} \{ \xi \in J_{K_n} v_n(x)  : \,\xi\cdot\nu_\Co(x)> \lambda +t\}$. Clearly,  the function $t \mapsto |\mathcal B^{\lambda+t}_{v_n} \cap B_r|$   is decreasing and Lipschitz continuous,   and at points of differentiability it holds 
\beq\label{eq:brutta2}
\frac{d}{d t} |\mathcal B^{\lambda+t}_{v_n} \cap B_r| = - \H^{N-1}(A_t \cap B_r) 
\eeq
for all $r>0$. 

We fix $x \in K_n$ and study the set 
\[
A_{t}(x) =  \{ \xi \in J_{K_n} v_n(x)  : \,\xi\cdot\nu_\Co(x)=  \lambda +t  \}. 
\]
Assume that $\xi \in A_{t}(x)$. In particular then $ \xi \in J_{K_n} v_n(x)  $.   We use Lemma \ref{lem:v-gamma} to deduce that  for every $s>0$  it holds 
$\xi + s \nu_\Co(x) \in J_{K_n} v_n(x)$. Therefore for every $s>0$ it holds 
\[
A_{t}(x) + s \{\nu_\Co(x)  \} \subset  A_{t+s}(x) .
\]
Since the map  $\xi \mapsto \xi  + s v_n(x)$ is an isometry, it holds $\H^{N-1}( A_{t}(x) \cap B_{1+t}) \leq \H^{N-1}( A_{t+s}(x) \cap B_{1+t+s}) $. Therefore   we deduce that the function 
\[
t \mapsto \H^{N-1}(A_t \cap B_{1+t})  = \sum_{x \in K_n}  \H^{N-1}( A_{t}(x) \cap B_{1+t})
\]
is non-decreasing. Integrating \eqref{eq:brutta2} we have 
\[
\int_{1}^2 \H^{N-1}(A_t \cap B_{1+t})\, dt   \leq -   \int_{1}^2 \frac{d}{dt} |\mathcal B^{\lambda+t}_{v_n} \cap B_3| \, dt  \leq  |B_3|.
\]
By the mean value theorem there is $\hat t \in [1,2]$ such that  $\H^{N-1}(A_{\hat{t}} \cap B_{1+\hat{t}})  \leq C. $
In turn, using the monotonicity obtained above, we have 
\[
 \H^{N-1}(A_t \cap B_{1})  \leq C \qquad \text{for all }\, t \in [0,1].
\]
The claim \eqref{eq:brutta} then follows by integrating \eqref{eq:brutta2} from $t=0$ to $t = \lambda'-\lambda$. 

Finally the statement of the lemma follows from \eqref{eq:bella2}, \eqref{eq:bella} and \eqref{eq:brutta} and letting $\lambda' \to \lambda$.   
\end{proof}

\begin{proposition}\label{th:ABP} Let $\Co$ be a closed convex set which satisfies \eqref{Coreg}  and assume  it satisfies  the $\lambda$-ABP property for some $\lambda\in (-1,1)$. Then, for every open set $\Om$  satisfying \eqref{Omreg}  and $|\Om|=v$ we have
\[
\J(\Om)\geq J_{\lambda,\textbf{H}}(B^\lambda[v]).   
\] 
Moreover, if  $\J(\Om)= J_{\lambda,\textbf{H}}(B^\lambda[v])$, then $\Om$ coincides with a solid spherical cap isometric to $ B^\lambda[v]$, supported on a facet of $\Co$. 
\end{proposition}
 \begin{proof}
 By scaling and Remark~\ref{rm:ABPscaled}, we may assume $v=|B^\lambda|$. The inequality follows from the argument from the introduction which we now repeat rigorously. Let $u :\bar \Omega \to \R$ be the variational  solution of the Neumann boundary problem  \eqref{eq:neumann-0} and denote its restriction to $\Gamma \subset \pa \Co$ by $u_\Gamma$.  Let $\mathcal A_u$ be the set  defined in \eqref{alambdau} and let $\mathcal B_{u_{\Gamma}}^\lambda$ be the set  defined in \eqref{blambdau2}. Denote also by $\Omega^+$ the set of points $x \in \Omega $ where $J_{\bar \Omega}u(x) \cap B_1$ is non-empty.  Then by Lemma~\ref{lem:u-gamma} it holds
\[
|\mathcal A_u \cap B_1|\geq |\mathcal B_{u_{\Gamma}}^\lambda \cap B_1|. 
\]
On the other hand, Definition \ref{def:ABP} of  the $\lambda$-ABP-property  and Lemma \ref{lem:ABP} yield
\[
|\mathcal B_{u_{\Gamma}}^\lambda \cap B_1| \geq |B^\lambda|.  
\]
Therefore we have 
\begin{equation}\label{eq:cabre}
\begin{split}
|\Omega|=|B^\lambda|\leq |\mathcal A_u \cap B_1| =  |\nabla u(\Om^+)| \leq \int_{\Omega^+}\det\nabla^2u \, dx \leq \int_{\Omega^+}\frac{(\Delta u)^N}{N^N}\, dx \leq \left( \frac{\J(\Omega )}{|\Omega| N}\right)^N |\Omega^+|. 
\end{split}
\end{equation}
The inequality then follows from  
\[
\J(B^\lambda) = N |B^\lambda| = N|\Omega| \geq N|\Omega^+|. 
\]

We now analyze the case of equality. Assume again without loss of generality that $v=|B^\lambda|$ and that $\Om$ is an open set satisfying \eqref{Omreg} for which $\J(\Om)= J_{\lambda,\textbf{H}}(B^\lambda)$. 

We begin by showing that  $\Om$ is connected. This follows from the fact that the isoperimetric function $v \mapsto J_{\lambda,\textbf{H}}(B^\lambda[v])$ is strictly convex.  Indeed, we argue by contradiction and assume there is a component  $\Om_1\subset\Om$, with $|\Om_1|=:v_1\in (0, |B^\lambda|)$ such that 
\[
\J(\Om)=\J(\Om_1)+\J(\Om\setminus \Om_1). 
\]
Then by the isoperimetric inequality that we just proved, and by  the strict concavity of the isoperimetric function it holds 
\beq\label{concav}
\J(\Om)=\J(\Om_1)+\J(\Om\setminus \Om_1)\geq J_{\lambda,\textbf{H}}(B^\lambda[v_1])+J_{\lambda,\textbf{H}}(B^\lambda[|B^\lambda|-v_1])>J_{\lambda,\textbf{H}}(B^\lambda),
\eeq 
which is a contradiction. Hence, $\Om$ is connected.

Let $u :\bar \Omega \to \R$ be the variational  solution of the Neumann boundary problem  \eqref{eq:neumann-0}. Since  $\J(\Om)= J_{\lambda,\textbf{H}}(B^\lambda)$, then  all the inequalities in \eqref{eq:cabre} become equalities and $|\Omega| = |\Omega^+|$. The equality in the arithmetic-geometric mean inequality means that 
\[
 \det\nabla^2u=\frac{(\Delta u)^N}{N^N}=1  \text{ and thus all the eigenvalues of $\nabla^2u$ are equal to 1 in }\Om.
\]
  Hence,    $\nabla^2u=I$ in $\Om$ and thus, by the connectedness of $\Om$, we infer that there exists $x_0\in \R^N$ such that, up to an additive constant, $u(x)=\frac{1}2|x-x_0|^2$ in $\Om$. Moreover, it follows from the definition of the set $\Omega^+$ that   $|\nabla u(x)|<1$ for all $x\in \Om^+$. Since the equalities in \eqref{eq:cabre} imply  that $|\Omega \setminus\Omega^+|= 0$, we have $|\nabla u(x)|=|x-x_0|\leq1$ for all $x\in \Om$. In other words, $\Om\subset B_1(x_0)$. In turn, since $1=\pa_\nu u(x)=(x-x_0)\cdot	\nu_{\Om}(x)$ for all $x\in\pa\Om\setminus \Co$,  we have necessarily $\pa\Om\setminus \Co\subset\pa B_1(x_0)$.

Let now $\hat x\in \Gamma\setminus \gamma$ and consider the half-space 
$$
H:=\{y\in \R^N:\, (y-x_0)\cdot\nu_{\Co}(\hat x)>\lambda\}.
$$
Since $-\lambda= \pa_\nu u(\hat x)=-(\hat x-x_0)\cdot\nu_{\Co}(\hat x)$, it follows that $\Co$ is contained in 
$\R^N\setminus H$ and that $\pa H$ is tangent to $\Co$ at $\hat x$. On the other hand, for any  nontangential direction $v\in \S^{N-1}$ such that $v\cdot \nu_{\Co}(\hat x)>0$, all the points of the half line $\hat x+tv$, $t>0$,  inside $B_1(x_0)$  are contained in $\Om$, since otherwise  the same half line would hit a point of $\pa \Om$ in $B_1(x_0)\\setminus \Co$, which is impossible since $\pa\Om\setminus \Co\subset\pa B_1(x_0)$.  Thus, we may conclude that $B_1(x_0)\cap H\subset \Om$. Since $B_1(x_0)\cap H$ is a spherical cap isometric to $B^\lambda$ (and since $|\Om|=|B^\lambda|$) we conclude that $\Om=B_1(x_0)\cap H$, thus concluding the proof of the proposition.

\end{proof}
\begin{remark}
Note that if  $N=3$ and the relative isoperimetric problem
\[
\min\{\J(E): E\subset\R^N\setminus \Co \text{ measurable, with }|E|=v\}
\]
has a solution $\Om$, then, due to the the regularity results of \cite{Taylor77}, see also \cite{De-PhilippisMaggi15},  $\Om$ satisfies \eqref{Omreg}, provided $\Co$ is of class $C^{1,1}$. Therefore if $\Co$ satisfies the $\lambda$-ABP property
 Theorem~\ref{th:ABP} yields the relative isoperimetric inequality for all sets of finite perimeter and fully characterizes the case of equality.
 \end{remark} 
We conclude this section by showing that the $0$-ABP property holds for all convex sets with nonempty interior.
\begin{proposition}\label{prop:0ABP}
Let $\Co\subset\R^N$ be a closed convex set of class $C^1$. Then $\Co$ satisfies the $0$-ABP property.
\end{proposition}
\begin{proof}
 Let $K=\{x_1, \dots, x_n\}$ be any discrete subset of $\pa\Co$ and let $v:K\to \R$.  Recall that $\displaystyle\bigcup_{i=1}^nJ_K v(x_i)=\R^N$ (see Remark~\ref{rm:ovv}) and that  the sets  $J_K v(x_i)$ are  a convex, in fact  a finite intersection of half-spaces,   and they have disjoint interiors. Moreover  Lemma~\ref{lem:v-gamma} implies that they have the  property 
\beq\label{recall}
\xi\in J_K v(x_i)\Longrightarrow \xi+t\nu_{\Co}(x_i)\in J_K v(x_i) \text{ for all $t>0$.}
\eeq

 
 Now, up to a set of Lebesgue measure zero, we may split $J_K v(x_i)=J_K v(x_i)^+\cup J_K v(x_i)^-$, where
\[
J_K v(x_i)^\pm:=\{\xi\in J_K v(x_i):\, \pm\, \xi\cdot\nu_{\Co}(x_i)>0 \}.
\]
Let us then fix   $\xi\in J_K v(x_i)^-$ and denote $\lambda = -\xi \cdot \nu_{\Co}(x_i)>0$. Then by \eqref{recall}  the symmetric point $\hat\xi = \xi + 2 \lambda \nu_{\Co}(x_i)$  belongs to $J_K v(x_i)^+$. Hence 
$|J_K v(x_i)^+\cap B_1|\geq \frac12 |J_K v(x_i)\cap B_1|$. In turn,
\beq\label{0ABPcrucial}
|\mathcal B^0_{v}|=\sum_{i=1}^n|J_K v(x_i)^+\cap B_1|\geq \frac12\sum_{i=1}^n |J_K v(x_i)\cap B_1|=\frac12 |B_1|,
\eeq
which is the desired estimate. From the arbitrariness of $K$, we conclude that $\Co$ satisfies the $0$-ABP property.
\end{proof}

\begin{remark}\label{rm:choeghomiritore}
By combining the previous proposition with Proposition~\ref{th:ABP} we recover an ABP-proof of the relative isoperimetric inequality outside convex sets obtained by Choe, Ghomi and Ritor\'e in \cite{ChGhoRi07}. Note that an ABP-argument for the same inequality has been already provided in \cite{LiuWangWeng}. However, in their argument our crucial estimate \eqref{0ABPcrucial} is replaced by a  geometric estimate based on normal cones and inspired by the techniques in \cite{ChGhoRi06}, see \cite[Proposition~2.4]{LiuWangWeng}.
\end{remark}
\section{The isoperimetric inequality outside convex cylinders}\label{three}

\subsection{Convex cylinders have the $\lambda$-ABP property}\label{sec:3.1}
 The purpose of this section is to prove the following: 
 \begin{theorem}\label{th:lambdaABP}
Let  $\Co$ be of the form $\C\times\R^{N-2}$, where $\C\subset\R^2$ is a closed  convex set of class $C^1$ with nonempty interior. Then $\Co$ satisfies the $\lambda$-ABP property for every $\lambda\in(-1,1)$.
 \end{theorem}
 
 In order to prove the above theorem we need to introduce some notation. 
  In the following we denote  generic vector $\xi\in\R^N$ by $(z,w)$ where $z\in\R^2$, $w\in\R^{N-2}$.  Given  $E\subset\R^N$ and $w\in\R^{N-2}$ we set
 $$
 E_w:=\{z\in\R^2:\,(z,w)\in E\}.
 $$
 In the following, given $\lambda\in\R$, $\nu\in\R^2$, we set
\beq \label{def:H-lambda}
 H^\lambda_\nu=\{z\in\R^2:\,z\cdot\nu>\lambda\}.
 \eeq
When $\nu=e_2$ we will simply write $H^\lambda$ instead of $H^\lambda_{e_2}$. We denote by $D_r\subset\R^2$ the open disk of radius $r$ centered at the origin and set $D=D_1$.

Let $K=\{x_1, \dots, x_n\}$ be any finite subset of $\Co$ and let $v:K\to \R$ be any function. The proof of Theorem~\ref{th:lambdaABP} will follow if we show that 
\beq\label{prop:lambdaABP1}
 |\mathcal B^\lambda_{v} \cap B_1|\geq|B^\lambda| ,
\eeq
where $\mathcal B^\lambda_{v}$ is defined in \eqref{blambdau}.
 To prove \eqref{prop:lambdaABP1} we will show that 
 \beq\label{prop:lambdaABP2}
  \H^2( (\mathcal B^\lambda_{v})_w \cap D_r)\geq \H^2((B^{\lambda,e_2})_w\cap D_r)
\eeq
 for all $\lambda\in(-1,1)$, $w\in\R^{N-2}$ with $|w|<1$ and $r = \sqrt{1-|w|^2}$. Since $\nu_{\Co}(x)\in\{\xi\in\R^N:\,\xi=(z,0)\}\approx\R^2$ for all $x\in\pa\Co$ 
 \begin{align*}
 (\mathcal B^\lambda_{v})_w \cap D_r
 &= \bigcup_{x\in K} \Big\{z\in D_{\sqrt{1-|w|^2}}:\, (z,w)\in J_K v(x),\,z\cdot\nu_\Co(x)> \lambda\Big \}, \\
 (B^{\lambda,e_2})_w &=H^\lambda\cap D_{\sqrt{1 -|w|^2}}.
\end{align*}
 Indeed, we shall prove a stronger inequality than \eqref{prop:lambdaABP2}, that is that for any $|w|<1$ 
 $$
\qquad\qquad  \H^2\Big(\bigcup_{x\in K} \{z\in D_r:\, (z,w)\in J_K v(x),\,z\cdot\nu_\Co(x)> \lambda \}\Big)\geq \H^2(H^\lambda\cap D_r), 
 $$
for all  $r\in(0,1)$. In turn, this inequality will follow from the inequality
 $$
 \qquad\qquad  \H^1\Big(\bigcup_{x\in K} \{z\in\pa D_\varrho : \, (z,w)\in J_K v(x),\,z\cdot\nu_\Co(x)> \lambda \}\Big)\geq \H^1(H^\lambda\cap\pa D_\varrho),  
$$
for all  $\varrho \in(0,1)$.   By rescaling, this last inequality is equivalent to
 \beq\label{prop:lambdaABP2.1}
  \H^1\Big(\bigcup_{x\in K}\{z'\in\pa D : \, (z',w')\in J_K v_\varrho(x),\,z'\cdot\nu_\Co(x)>\lambda' \}\Big)\geq\H^1(H^{\lambda'}\cap\pa D),
\eeq
 where we have set $w'=\frac{w}{\varrho}\in\R^{N-2}$, $\lambda'=\frac{\lambda}{\varrho}\in\R$, $u_\varrho=\frac{u}{\varrho}$.
 Note that since for all $w'$
 $$
 \bigcup_{x\in K}\Big \{z'\in\R^2: \, (z',w')\in J_K v_\varrho(x)\Big\}=\R^2,
 $$
 \eqref{prop:lambdaABP2.1} is trivially satisfied if $\lambda'\geq1$ or $\lambda'\leq-1$. Therefore, we are ultimately bound to show that for any function $v:K\to\R$, any $w\in\R^{N-2}$ and any $\lambda\in(-1,1)$
 \beq\label{prop:lambdaABP3}
  \H^1\Big(\bigcup_{x\in K}\{z\in\pa D : \, (z,w)\in J_K v(x),\,z\cdot\nu_\Co(x)>\lambda \})\geq \H^1(H^\lambda\cap\pa D).
  \eeq
  The rest of the section is devoted to the proof of \eqref{prop:lambdaABP3}.  

 Given $w\in\R^{N-2}$,  for all $i=1,\dots,n$ we define the {\it cell} $A_i$ and the normal $\nu_i$ associated with it as
 \beq\label{prop:lambdaABP3.1}
 A_i=\{z\in\R^2:\,(z,w)\in J_K v(x_i)\}  \quad \text{and} \quad  \nu_i=\nu_{\Co}(x_i).
\eeq
Recall that each $A_i$ is convex and more precisely is obtained as a  finite intersection  of half planes and that the $A_i$'s have pairwise disjoint interiors and their union is the whole $\R^2$. 
Moreover, from Lemma~\ref{lem:v-gamma} we have that  if $\xi\in J_K v(x_i)$ then $\xi+t\nu_i\in J_\Gamma(x_i)$ for all $t>0$. Therefore 
 \beq\label{prop:lambdaABP4.1}
z\in A_i, \, t>0\quad\Longrightarrow\quad z+t\nu_i\in A_i.
\eeq
In what follows we say $A_i$ and $A_j$ are {\it neighboring cells}  if $\H^1(\pa A_i\cap \pa A_j)>0$. 

We prove  \eqref{prop:lambdaABP3} by induction and to this aim  we need the following  deleting procedure. 
 Remove a point $x_j$ from $K$, set $\widetilde K=K\setminus{x_j}$ and denote by $\tilde v$ the restriction of $v$ to $\widetilde K$ and by $J_{\widetilde K}\tilde v(x_i)$ the corresponding subdifferentials. Finally, define  $\widetilde A_i$, for $i\neq j$  as in \eqref{prop:lambdaABP3.1} with $K$ and $v$ replaced by $\widetilde K$ and $\tilde v$ respectively (and with $w$ unchanged). Clearly,  $J_{K}v (x_i)\subset J_{\widetilde K}\tilde v(x_i)$, hence $A_i\subset \widetilde A_i$, for all $i\neq j$. However, since also the interiors of the $\widetilde A_i$'s are mutually disjoint and convex  and $\bigcup_{i\neq j}{\widetilde A}_i=\R^2$, we have necessarily that
  \beq\label{prop:lambdaABP4.2}
\qquad\qquad \widetilde A_i=A_i \quad \text{if $A_i$ and $A_j$ are not neighboring cells}.
\eeq
We will refer to this procedure as {\it deleting a cell}. Note that if $A_j\cap D=\emptyset$, we may delete $A_j$ and this deletion does not affect the intersections of the remaining $A_i$'s  with $D$. Therefore, without loss of generality we may assume that all the $A_i$'s intersect $D$. 

We will say that $A_i$ is a {\it disconnecting cell} if $D\setminus A_i$ is disconnected, see Figure~\ref{fig2}. We will also say that a disconnecting cell $A_i$ is {\it extremal} if at least one of the two connected components of $\overline{D\setminus A_i}$ contains no disconnecting cells.
\begin{figure}
\begin{tikzpicture}
\begin{scope}[yscale=.5,xscale=.5]
\clip (-8, -6) rectangle (8, 6); 

\draw (0,0) circle (5); 
\filldraw[fill=black] (0,0) circle (.05);

\draw (0.5,0) node {$0$}; 

\draw[dashed] (1.5,5.5) -- (1,-5.5); 
\draw[dashed] (-1.7,5.5) -- (-0.3,-5.5); 

\draw[dashed] (-2,5.5) -- (-2.5,-5); 
\draw[dashed] (2.6,5) -- (3.5,-4.5); 
\draw[dashed] (3,1) -- (5.2, 1.5); 
\draw[dashed] (3.1,-0.5) -- (5.5, -1);

\draw (0.5,-2.5) node {$A_1$}; 
\draw (2,0.5) node {$A_2$}; 
\draw (-1.7,0.5) node {$A_3$};              
\draw (-3.7,-1) node {$A_5$};   

\draw (4,-2) node {$A_4$};   
\draw (4,0) node {$A_6$}; 
\draw (3.5,2) node {$A_7$};            

\draw[loosely dotted, fill=white!60!black,opacity=0.65, very thin]   (1.5,4.8) -- (1,-4.9) -- (-0.38,-5) -- (-1.6,4.76) -- cycle;               
\path[loosely dotted, fill=white!60!black,opacity=0.65, very thin] 
(1.5,4.8)  arc[start angle =73, end angle =108.7, radius = 5]  -- cycle;

\draw[loosely dotted, fill=white!60!black,opacity=0.2, very thin]   (1.5,4.8) -- (2.7,4.2) -- (3.45,-3.6) -- (1,-4.9) -- cycle    ;   
\path[loosely dotted, fill=white!60!black,opacity=0.2, very thin] 
(1,-4.9) arc[start angle =282, end angle =314, radius = 5]  -- cycle;

\draw[loosely dotted, fill=white!60!black,opacity=0.2, very thin]   (-2.05,4.55)  -- (-1.6,4.76) -- (-0.38,-4.98)   -- (-2.45,-4.35) -- cycle    ;      
\path[loosely dotted, fill=white!60!black,opacity=0.2, very thin] 
(-2.45,-4.35)  arc[start angle =240, end angle =266, radius = 5]  -- cycle;

\end{scope}
\end{tikzpicture}
\caption{An example of a configuration with seven cells. The cells $A_1, A_2$ and $A_3$ are disconnecting and $A_2$ and $A_3$ are also extremal.}\label{fig2}
\end{figure}
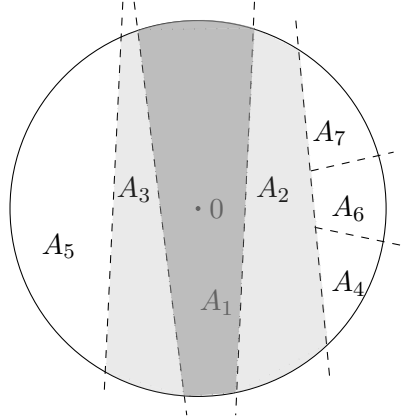
 
We have the  following lemma.
\begin{lemma}\label{lm:discon}
Let $A_i$ be a disconnecting cell. Then  either $A_i$ is the only the disconnecting cell (hence extremal) or  there are at  least two extremal disconnecting cells. 
\end{lemma}
\begin{proof}
 Denote by $C_1$ and $C_2$ the components of $\overline {D\setminus A_i}$ and assume that one of them, say $C_1$, contains a disconnecting cell $A_j$. Consider the component of  $\overline {D\setminus A_j}$ that does not contain $C_2$. Either this component contains no disconnecting cells (which makes $A_j$ extremal) or we iterate this procedure until we get after finitely many steps an extremal disconnecting cell contained in $C_1$. Now, if $A_i$ is not extremal, we apply the same procedure in $C_2$ to find a second extremal disconnecting cell. 
\end{proof}

We introduce  the following useful quantities: for every $i=1,\dots, n$, we set  (see Figure~\ref{fig1})
\beq\label{entry}
\begin{split}
&\ell_i:=\min\{z\cdot \nu_i:\, z\in A_i\cap \pa D\} \quad\text{(entry value of $A_i$)}\\
&m_i:=\max\{z\cdot \nu_i:\, z\in A_i\cap \pa D\} \quad\text{(exit value of $A_i$)},
\end{split}
\eeq 

\begin{figure}
\begin{tikzpicture}
\begin{scope}[yscale=.5,xscale=.5]
\clip (-8, -6) rectangle (8, 6); 

\draw (0,0) circle (5); 
\filldraw[fill=black] (0,0) circle (.05);

\draw (0.5,0.5) node {$0$}; 
\draw[dashed] (-6,1) -- (3,1); 
\draw[dashed] (3,1) -- (3,5);   
\draw[dashed] (3,1) -- (6,-1);   
\draw (4,1.3) node {$A_3$}; 
\draw (-2,2.5) node {$A_2$}; 
\draw (1,-3) node {$A_1$}; 

\draw (5.2,-3.2) node {$\nu_1$};
\draw[->] (4.5,-3.5) -- (5.2,-4); 

\draw[->] (0, -6) -- (0,6); 

\draw[->] (4.5, 4) -- (4.5,4.8); 
\draw (5.3 ,4.3) node {$\nu_3$};

\draw[->] (-5, 2) -- (-5.9,2.2);
\draw (-5 ,3) node {$\nu_2$};         

\draw[densely dotted]  (3,4) -- (-0.15,4); 
\draw (-0.6 ,4) node {$m_3$};
    
\draw[densely dotted]  (5,-0.33) -- (-0.15,-0.33); 
\draw (-0.6 ,-0.33) node {$\ell_3$};

\draw[loosely dotted, fill=white!60!black,opacity=0.65, very thin]   (3,1)-- (3,4) -- (5,-0.33)  -- cycle    ;

\path[loosely dotted,fill=white!60!black,opacity=0.65, very thin] 
(5,-0.315) arc[start angle = -3, end angle =52.8, radius = 5]  -- cycle;

\draw[loosely dotted,fill=white!60!black,opacity=0.25, very thin]  (-4.9,1) -- (3,1)-- (3,4)  -- cycle  ;
\path[loosely dotted, fill=white!60!black,opacity=0.25, very thin] 
(3,4)  arc[start angle = 52.5, end angle =169, radius = 5]  -- cycle;

\draw[dotted, fill=white!40!black,opacity=0.05, very thin]  (-4.9,1) -- (3,1)-- (5,-0.33) -- cycle  ;
\path[dotted, fill=white!40!black,opacity=0.05, very thin] 
(-4.9,1)  arc[start angle =169, end angle =355.5, radius = 5]  -- cycle;
\end{scope}
\end{tikzpicture}

\caption{An example of a configuration with three cells with the associated normals $\nu_1, \nu_2$ and $\nu_3$. The entry value of the third cell is $\ell_3$ and the exit value is $m_3$.}\label{fig1}
\end{figure}
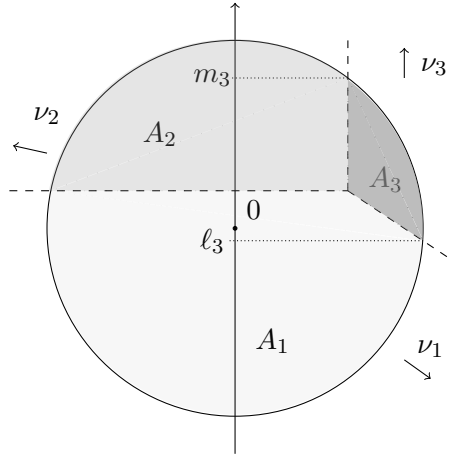

Note that by \eqref{prop:lambdaABP4.1} it is easily checked that if $0\in A_i$ then $m_i=1$ and that $m_j> 0$ for all $j$.
 By  reordering the indices  in what follows we may assume that $0\in A_1$ and that 
 \beq\label{ordering}
 1=m_1\geq m_2\geq\dots..\geq m_n.
 \eeq
 Note also that if $A_i$ is not disconnecting, then it holds  that 
 \beq\label{pointprop}
 \lambda\in [\ell_i, m_i]\Longrightarrow   \text{there exists }\,  z\in A_i\cap \pa D\text{ s.t. } z\cdot \nu_i=\lambda. 
  \eeq
  Indeed, if $A_i$ is not disconnecting, then   the intersection of $A_i$ with $\pa D$ is a connected arc and therefore the function $z\mapsto z\cdot  \nu_i$ takes all values in $ [\ell_i, m_i]$ along such arc. 

 We also define 
  $$
 l:=\max\{z\cdot \nu_1:\, z\in \pa A_1\cap \pa D\}
 $$
 and note that $\ell_1\leq l\leq1$. Moreover, if  $l<1$ and $\lambda\in (l, 1)$ (recall that $m_1=1$), then the segment $D\cap \pa H^\lambda_{\nu_1}$ does not intersect $\pa A_1$ and therefore 
 \beq\label{twopoint0}
 \H^0(\pa D\cap \pa H^\lambda_{\nu_1}\cap A_1)=2\quad \text{for all }\lambda\in (l, 1).
 \eeq
 \begin{lemma}\label{lm:crucial}
 It holds $m_2\geq l$. 
 \end{lemma}
 \begin{proof}
We may assume that $l \geq 0$ since otherwise the claim is trivially true as $m_2 >0$. Without loss of generality we may assume $\nu_1 =e_2$. Let $\bar z\in \pa A_1$ be such that $|\bar z|=\dist(0, \pa A_1)=:d$,  see Figure~\ref{fig3}.  We claim that 
 \beq\label{claim}
 d\leq \sqrt{1-l^2}.
 \eeq
 To this aim, let $\tilde z\in \pa A_1\cap\pa D$ be such that $l=\tilde z\cdot e_2$. Let $H$ be a (closed) half-plane such that $\tilde z\in \pa H$ and $A_1\subset H$. Since $A_1$ contains the origin and $A_1\subset H$ it holds $d = \dist(0, \pa A_1)\leq  \dist(0, \pa H)$. Let $\omega$ be the unit vector normal to $\pa H$ pointing towards $A_1$, i.e., $H = \{z : (z - \bar z) \cdot \omega \geq  0 \} $. Then using the property
 \eqref{prop:lambdaABP4.1} it holds $ \bar z +t e_2 \in A_1 \subset H $ for all $t>0$, which  implies  $\omega\cdot e_2\geq 0$. In turn, since $l\ge 0$, this implies that $\dist(0, \pa H)$ is less or equal than the distance of $\tilde z$ to the $e_2$-axis, which is given by $\sqrt{1-l^2}$ (see Figure~\ref{fig3}). Hence, the claim \eqref{claim} follows.
 
 \begin{figure}
\begin{tikzpicture}
\begin{scope}[yscale=.5,xscale=.5]
\clip (-8, -6) rectangle (8, 6); 

\draw (0,0) circle (5); 
\filldraw[fill=black] (0,0) circle (.05);

\draw[densely dotted] (0,0) circle (0.88);

\draw (0.5,0) node {$0$}; 

\draw[dashed] (-3,4.5) -- (-1.6,-0.5); 
\draw[dashed](-1.6,-0.5) -- (0,-1.5); 
\draw[dashed] (0,-1.5)-- (4,4); 
\draw[dashed] (0,-1.5)-- (2,-5.5); 

\draw (1.5,2) node {$A_1$}; 
\draw (3.5,-2) node {$A_i$};

\draw (0.7,5.3) node {$\nu_1$};
\draw[->] (0,0) -- (0,5.8);

\draw (5.5,2.4) node {$\nu_i$};
\draw[->] (0,0) -- (5,2.7);

\filldraw[fill=black] (0.7,-0.56) circle (.05);
\draw (1,-1) node {$\bar z$};

\draw[densely dotted]  (-2.9,4.05) -- (0.15,4.05); 
\draw (0.4,4) node {$l$};

\draw[densely dotted]  (3.62,3.42)  -- (4.3,2.1); 
\draw (4.1,1.7) node {$m_i$};

\draw[dotted]  (0.7,-0.56) -- (4.83,1.44); 
\filldraw[fill=black] (4.78,1.36) circle (.05);

\draw  (-3.5,4.05) node {$\tilde  z$};

\draw (6.3,1.4) node {$\hat t\nu_i + \bar z$};

\end{scope}
\end{tikzpicture}
\caption{Situation as in the proof of Lemma~\ref{lm:crucial}. The point $\bar z$ is chosen as the closest to the origin on $\pa A_1$. The cell $A_i$ is neighbour to $A_1$ such that $\bar z \in \pa A_i$.} \label{fig3}
\end{figure}
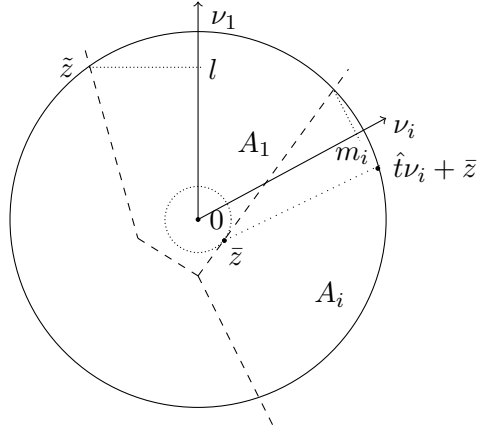

Since  $\bar z\in \pa A_1$ there is a neighboring  cell, say $A_i$, such that $\bar z\in \pa A_i$. Then the choice of $\bar z$ implies that $d =|\bar z|=\dist(0, A_i)$. On the other hand by the property 
\eqref{prop:lambdaABP4.1}  it holds $t\nu_i + \bar z \in  A_i$ for every $t>0$. Therefore $|t\nu_i + \bar z|^2 \geq|\bar z|^2$ for $t>0$ from which we deduce $\bar z \cdot \nu_i \geq 0$. 

Denote by $\hat t \in [0,1]$ the value such that $\hat t\nu_i + \bar z \in \partial D$. Then 
\[
1 = |\hat t\nu_i + \bar z|^2 =\hat t^2 + 2 \hat t \bar z \cdot \nu_i + |\bar z|^2 =\hat t^2 + 2 \hat t \bar z \cdot \nu_i +d^2.
\]
By using this and  \eqref{claim} we have 
 \beq\label{claim2}
\hat t^2 + 2 \hat t \bar z \cdot \nu_i  \geq l^2.
 \eeq
On the other hand, since  $\hat t\nu_i + \bar z \in  A_i$,  the definition of the exit value  implies 
\[
m_i \geq (\hat t\nu_i + \bar z) \cdot \nu_i = \hat t +  \bar z  \cdot \nu_i.  
\]
Since  $\bar z \cdot \nu_i \geq 0$, this implies 
\[
m_i^2 \geq  \hat t^2 +  2 \hat t \bar z  \cdot \nu_i + (\bar z  \cdot \nu_i)^2 \geq \hat t^2 +  2 \hat t \bar z  \cdot \nu_i 
\]
and the inequality of the lemma follows from  \eqref{claim2} and from the fact that $m_2 \geq m_i$.
 \end{proof}
 \begin{proof}[\textbf{Proof of Theorem~\ref{th:lambdaABP}}] We let  $K =\{x_1, \dots, x_n\}\subset\Co$ and recall  that we just need to show \eqref{prop:lambdaABP3}. 
 We proceed by defining 
\[
 \varphi_{K}(\lambda)=\sum_{i=1}^n\H^1(A_i\cap H^{\lambda}_{\nu_i}\cap\pa D), \quad \lambda\in[-1,1],
 \]
 where the cells $A_i$ and the associated normal $\nu_i$  are defined in \eqref{prop:lambdaABP3.1}  (for a fixed $w\in\R^{N-2}$) and the half-spaces $H^{\lambda}_{\nu_i}$ in \eqref{def:H-lambda}. Recall that the $A_i$'s have mutually disjoint interiors, they are finite intersections of half-spaces and $\bigcup_iA_i=\R^2$. 
Setting also $\varphi_H(\lambda)=\H^1(H^\lambda\cap\pa D)$, the  inequality \eqref{prop:lambdaABP3} can be rewritten as
 \beq\label{prop:lambdaABP4}
 \varphi_K(\lambda)\geq\varphi_H(\lambda)  \quad\text{for all } \lambda\in (-1,1).
 \eeq
 Note that 
 \beq\label{endpoints}
 \begin{split}
 & \varphi_K(-1)=2\pi=\varphi_H(-1), \\
& \varphi_K(1)=0=\varphi_H(1).
  \end{split}
 \eeq
 We prove the inequality by induction on the cardinality $n$ of $K$. Note that the conclusion is trivially true if $n=1$. 
 Therefore we assume the statement to be true for all $k\leq n-1$.

Note that a simple application of the area formula implies that $\varphi_K$ is locally Lipschitz continuous in $(-1,1)$ and that for a.e. $\lambda\in(-1,1)$
\beq\label{der}
\begin{split}
\varphi'_K(\lambda)&=-\sum_{i=1}^n\int_{A_i\cap\pa H^{\lambda}_{\nu_i}\cap\pa D}\frac{1}{\sqrt{1-(z\cdot\nu_i})^2}\,d\H^0(z)\\
&=-\frac{1}{\sqrt{1-\lambda^2}}\sum_{i=1}^n\H^0(A_i\cap\pa H^{\lambda}_{\nu_i}\cap\pa D),
\end{split}
\eeq
while
\beq\label{der2}
\varphi'_H(\lambda)=-\frac{2}{\sqrt{1-\lambda^2}}.
\eeq

Let us denote 
 \beq\label{maxentry}
 \Lambda:=\max_{i\in \{1,\dots, n\}}\ell_i,
 \eeq
  where the $\ell_i$'s are the {\em entry values} defined in \eqref{entry}. We begin by proving the inequality \eqref{prop:lambdaABP4} for all $\lambda \in [-1,\Lambda]$. Clearly we may assume that  $\Lambda>-1$. 

Let $j$ be  such that $\ell_j=\Lambda$ and 
 $\lambda\in[-1,\Lambda)$. Since $\lambda<\ell_j$, by definition of the entry value  we have $A_j\cap H^{\lambda}_{\nu_j}\cap\pa D=A_j\cap\pa D$, hence
$$
\varphi_{K}(\lambda)=\sum_{i\neq j}\H^1(A_i\cap H^{\lambda}_{\nu_i}\cap\pa D)+\H^1(A_j\cap\pa D).
$$
Denote by $\widetilde A_i$, $i\neq j$, the new partition obtained by deleting the cell $A_j$ (see the definition of this deleting procedure before \eqref{prop:lambdaABP4.2}).
By the induction assumption we have
$$
\sum_{i\neq j}\H^1(\widetilde A_i\cap H^{\lambda}_{\nu_i}\cap\pa D)\geq\varphi_H(\lambda) \qquad \text{for all } \,\lambda\in[-1,1].
$$
In turn, we conclude that if $\lambda\in[-1,\Lambda)$
\begin{align*}
\varphi_{K}(\lambda)
&=\sum_{i\neq j}[\H^1(A_i\cap\pa D)-\H^1((A_i\cap\pa D)\setminus H^{\lambda}_{\nu_i})]+\H^1(A_j\cap\pa D) \\
&=\sum_{i=1}^n\H^1(A_i\cap\pa D)-\sum_{i\neq j}\H^1((A_i\cap\pa D)\setminus H^{\lambda}_{\nu_i})\\
&\geq 2\pi-\sum_{i\neq j}\H^1((\widetilde A_i\cap\pa D)\setminus H^{\lambda}_{\nu_i})=\sum_{i\neq j}\H^1(\widetilde A_i\cap\pa D \cap H^{\lambda}_{\nu_i})\geq \varphi_H(\lambda),
\end{align*}
where the first inequality follows from the fact the $A_i$'s cover the whole plane and the fact that $A_i\subset \widetilde A_i$ for $i\neq j$.
This proves \eqref{prop:lambdaABP4} for $-1\leq\lambda<\Lambda$, hence  also for $1\leq\lambda\leq\Lambda$ by continuity.

In order to prove \eqref{prop:lambdaABP4} for $\lambda \in (\Lambda, 1)$ we distinguish two cases.

\vskip 0.2cm\noindent
 \textbf{Case 1.} We first assume that there are no disconnecting cells.  
 
We will show that 
\beq\label{twopoint}
\sum_{i=1}^n\H^0(A_i\cap\pa H^{\lambda}_{\nu_i}\cap\pa D)\geq 2 \quad\text{for }\lambda\in (\Lambda, 1).
\eeq
Indeed, note that once this is established, then  by \eqref{der} and \eqref{der2}  we get $\varphi'_K(\lambda)-\varphi'_H(\lambda)\leq 0$ for $\lambda \in (\Lambda, 1)$. That is, $\varphi_K-\varphi_H$ is  non-increasing in $ (\Lambda, 1)$. Recalling that $\varphi_K(1)-\varphi_H(1)=0$ by \eqref{endpoints}, the conclusion would follow. 

Let us therefore focus on the proof of  \eqref{twopoint}. Recall that $A_1$ is the cell containing the origin of $\R^2$ and that the cells are indexed in such a way that the {\em exit values} defined in \eqref{entry} satisfy \eqref{ordering}. Recall also that by  \eqref{twopoint0}
\[
 \H^0(\pa D\cap \pa H^\lambda_{\nu_1}\cap A_1)=2 \quad\text{for }\lambda\in (l, 1)
\]
and since $A_1$, $A_2$ are not disconnecting it holds by \eqref{pointprop} that 
\[
\H^0(\pa D\cap \pa H^\lambda_{\nu_i}\cap A_i)\geq 1 \quad\text{for }\lambda\in (\ell_i, m_i), \, i=1,2.
\]
The conclusion \eqref{twopoint} now follows recalling that by Lemma~\ref{lm:crucial} $m_2\geq l$.

\vskip 0.2cm\noindent
\textbf{Step 2}. Here we assume that there exists at least one disconnecting cell. 
By Lemma~\ref{lm:discon} we either have only one disconnecting cell or we can find two distinct extremal cells $A_{j_1}$, $A_{j_2}$. Let us assume from now on that the second case holds, the other one being analogous (in fact simpler). 
Let $C_1$ be the component of $\overline{D\setminus A_{j_1}}$, which does not contain $A_{j_2}$, and similarly let $C_2$ be the component of $\overline{D\setminus A_{j_2}}$, which does not contain $A_{j_1}$. Clearly $C_1$ and $C_2$ are disjoint. Let now $J_1$, $J_2$ be two disjoint set of indices such that 
\[
C_1=\bigcup_{i\in J_1}(A_i\cap\overline D)\quad\text{and}\quad C_2=\bigcup_{i\in J_2}(A_i\cap\overline D).
\]
(To illustrate the situation the reader may consider the case in Figure~\ref{fig2}, where  $A_{j_1}= A_2$, $A_{j_2} = A_3$,  $C_1 = (A_4 \cup A_6 \cup A_7) \cap \overline D$ and $C_2 = A_5\cap \overline D$.) Define also for $j=1,2$
$$
M_j:=\max_{i\in J_j}m_i\quad\text{and}\quad \overline M:=\min\{M_1, M_2\}.
$$

We claim that 
\beq\label{calim}
 \varphi_K(\lambda)\geq\varphi_H(\lambda)  \quad\text{for all } \lambda\in[\overline M,1].
\eeq
Indeed, without loss of generality we may assume $\overline M=M_1$ and $M_1<1$.   By deleting  the cells $A_i$  for all  $i\not\in J_1$ and $i\neq j_1$, we generate a new partition of the form $\{A_i\}_{i\in J_1}\cup\{\tilde A\}$, with $\tilde A\supset A_{j_1}$. Note that we used the fact that by this deleting procedure the cells in $C_1$ are not affected due to \eqref{prop:lambdaABP4.2}. Note also that $\tilde A$ contains the origin and satisfies \eqref{prop:lambdaABP4.1} with respect to $\nu_{j_1}$. Indeed, otherwise the origin would be in $C_1$ and this in turn would imply that $M_1=1$. Let us define 
\[
\tilde l:=\max\{z\cdot \nu_{j_1}:\, z\in \pa \tilde A\cap \pa D\}.
\]
Then, it holds 
$\{z\in D:\, z\cdot \nu_{j_1}>\tilde l\}\subset \tilde A$
and thus
\beq\label{C1}
C_1\subset \{z\in \overline D:\, z\cdot \nu_{j_1}\leq \tilde l\}.
\eeq
Moreover,  by Lemma~\ref{lm:crucial} we have
\beq\label{tilde l}
M_1\geq \tilde l.
\eeq

Next we need another deleting procedure from the original partition $\{A_i\}_{i=1, \dots, n}$; namely, this time we delete all the cells $A_i$ with $i\in J_1$. Using again \eqref{prop:lambdaABP4.2}, we see that newly generated partition is of the form 
$\{A_i\}_{i\not \in J_1}\cup\{\tilde A_{j_1}\}$, where $\tilde A_{j_1}=A_{j_1}\cup C_1$.  By the induction assumption it holds
\[
\sum_{i\not \in  J_1}\H^1( A_i\cap H^{\lambda}_{\nu_i}\cap\pa D)+\H^1(\widetilde A_{j_1}\cap H^{\lambda}_{\nu_{j_1}}\cap\pa D)\geq\varphi_H(\lambda) ,
\]
for all $\lambda\in[-1,1]$. Note now that by \eqref{C1} we have $\H^1(\widetilde A_{j_1}\cap H^{\lambda}_{\nu_{j_1}}\cap\pa D)=\H^1( A_{j_1}\cap H^{\lambda}_{\nu_{j_1}}\cap\pa D)$ for $\lambda>\tilde l$. Therefore, for all $\lambda>\tilde l$ we have
\[
\begin{split}
\varphi_{K}(\lambda)&\geq \sum_{i\not \in  J_1}\H^1( A_i\cap H^{\lambda}_{\nu_i}\cap\pa D)+\H^1(\widetilde A_{j_1}\cap H^{\lambda}_{\nu_{j_1}}\cap\pa D)\\
&=\sum_{i\not \in  J_1}\H^1( A_i\cap H^{\lambda}_{\nu_i}\cap\pa D)+\H^1( A_{j_1}\cap H^{\lambda}_{\nu_{j_1}}\cap\pa D)
\geq \varphi_H(\lambda).
\end{split}
\]
 Recalling \eqref{tilde l}, claim \eqref{calim} follows.  

Recalling the beginning of the proof and \eqref{calim}, we have proved the inequality  \eqref{prop:lambdaABP4} for $\lambda \in [-1, \Lambda] \cup [\overline M,1]$, where $\Lambda$ is defined in \eqref{maxentry}. To reach the conclusion, we need to fill the gap between $\Lambda$ and $\overline M$ (assuming without loss of generality $\Lambda<\overline M$). To this aim, we will show that 
\beq\label{twopointbis}
\sum_{i=1}^n\H^0(A_i\cap\pa H^{\lambda}_{\nu_i}\cap\pa D)\geq 2 \quad\text{for }\lambda\in (\Lambda, \overline M).
\eeq
Once the claim is established, arguing as in the previous case we may deduce  that $\varphi_K-\varphi_H$ is non-increasing in $(\Lambda, \overline M)$, and the conclusion then follows from \eqref{calim}.

Let now  $i_1\in J_1$, be such that $m_{i_1}=M_1$. Since $A_{i_1}$ is  non disconnecting, then by \eqref{pointprop} we have 
\[
\H^0(\pa D\cap \pa H^\lambda_{\nu_{i_1}}\cap A_{i_1})\geq 1 \quad\text{for }\lambda\in (\ell_{i_1}, M_1).
\]
Similarly, there exist $i_2\in J_2$ such that  $m_{i_1}=M_1$ and 
\[
\H^0(\pa D\cap \pa H^\lambda_{\nu_{i_2}}\cap A_{i_2})\geq 1 \quad\text{for }\lambda\in (\ell_{i_2}, M_2).
\]
The previous two inequalities yield \eqref{twopointbis}, recalling that $\ell_{i_1}$ and $\ell_{i_2}$ are less than or equal $\Lambda$, while $M_i\geq \overline M$ for $i=1, 2$. This concludes the proof of \eqref{prop:lambdaABP3} and in turn of Theorem~\ref{th:lambdaABP}.
 \end{proof}
 
\subsection{Proof of the main result}
In this section we finally prove Theorem~\ref{thm1}. We will need the following approximation lemma, whose proof  is given in the Appendix.

 To this aim we recall we recall that a sequence $\{C_n\}$ of closed sets of $\R^N$ converge in the {\em Kuratoswki sense} to a closed set $ C$ if the following conditions are satisfied:
\begin{itemize}
\item[(i)] if $x_n\in C_n$ for every $n$, then any limit point of $\{x_n\}$ belongs to $ C$;
\item[(ii)] any $x\in C$ is the limit of a sequence $\{x_n\}$ with $x_n\in C_n$.
\end{itemize}
One can easily see that $C_n\to C$ in the sense of Kuratowski if and only if dist$(\cdot, C_n)\to$ dist$(\cdot, C)$ locally uniformly in $\R^N$. 

\begin{lemma}\label{appro}
Let $\Co\subset\R^N$ be of the form $\C\times\R^{N-2}$, where $\C\subset\R^2$ is a closed convex set  of the plane with nonempty interior. Let $E\subset\R^N\setminus\Co$ be a   set of finite perimeter. Then there exist a sequence of  closed convex sets $\Co_n$  of the form $\Co_n=\C_n\times\R^{N-2}$, $\C_n\subset\R^2$,  which satisfy \eqref{Coreg}  and a sequence of open sets $\Om_n\subset\R^N\setminus\Co_n$ which satisfy \eqref{Omreg},  in fact $\pa \Co_n$ and  $\Sigma_n:=\pa\Om_n\setminus\Co_n$ are $C^\infty$, such that  the following hold:
\vskip 0.2cm \par\noindent
{\rm (i)}\, $\Co_n\to\Co$ in the Kuratowski sense, with $\Co\subset \Co_n$ for all $n$;
\vskip 0.2cm \par\noindent
{\rm (ii)}\, $|\Om_n\triangle E|\to0$ as $n\to\infty$ and  $\pa\Om_n\subset\big\{x:\,{\rm dist}(x,\pa E)<\frac{1}{n}\big\}$;
\vskip 0.2cm \par\noindent
 {\rm (iii)}\, $P(\Om_n;\R^N\setminus\Co_n)\to P(E;\R^N\setminus\Co)$;
 \vskip 0.2cm \par\noindent
  {\rm (iv)}\, $\H^{N-1}(\pa\Om_n\cap\pa\Co_n)\to\H^{N-1}(\pa^*E\cap\pa\Co)$.
  
  Moreover, if $E$ coincides with a bounded open set of finite perimeter $\Om\subset \R^N\setminus \Co$ such that 
$\pa\Om\setminus \Co$ is smooth, then we may construct the approximating sets in  such a way that in addition it holds: 
\vskip 0.2cm \par\noindent
  {\rm (v)}\, $\Om_n\setminus \Co_n=(\Om)_{\e_n}\setminus  \Co_n$ for a suitable sequence $\e_n\to 0$, where $(\Om)_{\e_n}$ denotes the $\e_n$-neighborhood of $\Om$.  
   \end{lemma} 

\begin{proof}[\textbf{Proof of Theorem~\ref{thm1}}] We begin by observing that by Proposition~\ref{th:ABP} and Theorem~\ref{th:lambdaABP} it follows that the inequality \eqref{main1} holds if $\Co$ and $\Omega\subset\R^N\setminus\Co$ satisfy \eqref{Coreg}   and   \eqref{Omreg} respectively.  In the general case, the inequality \eqref{main1} follows by approximation, see Lemma~\ref{appro}.

We now analyse the case of equality. Assume that $E$ is a set of finite perimeter  with $|E|=v$ for which equality in \eqref{main1} holds. Thus, in particular, $E$ is a minimizer of the isoperimetric problem \eqref{def:min-prob}. Without loss of generality, we may assume $v=|B^\lambda|$. 

We now  apply to  $E$ the Steiner symmetrisation of codimension $N-2$ with respect to the plane containing the section $\C$ of the cylinder $\Co$, see \cite{BCF13}. Recall that we write $x=(z, w)$, with $z\in \R^2$ (the plane containing $\C$) and $w\in \R^{N-2}$. For any $z\in \R^2$ we let $E_z:=\{w\in \R^{N-2}:\, (z,w)\in E\}$ be the corresponding $(N-2)$-slice of  $E$ and define the function 
\beq\label{mslice}
m(z):=\H^{N-2}(E_z). 
\eeq
and, correspondingly, we let $r(z)$ be the radius of the $(N-2)$-dimensional ball having $(N-2)$-measure equal to $m(z)$. Setting also  
\[
\pi(E):=\{z\in \R^2:\, m(z)>0\},
\]
we can define the Steiner symmetral of $E$ as
\[
E^S:=\{(z, w)\in \R^{N}:\, z\in \pi (E) \text{ and }|w|<r(z)\}.
\]
Clearly $|E^S|=|E|$.  Since $\pa^*E^S\cap \pa \Co$ is equivalent to $(\pa^*E\cap \pa \Co)^S$ we also have $\H^{N-1}(\pa^*E^S\cap \Co)=\H^{N-1}(\pa^*E\cap \Co)$. Finally,   
$P(E^S; \R^N\setminus \Co)\leq P(E; \R^N\setminus \Co)$ by \cite[Theorem~1.1]{BCF13} and thus $\J(E^S)=\J(E)=\J(B^\lambda)$.  In particular, also $E^S$ is a minimizer of   \eqref{def:min-prob}. Then $E^S$  is equivalent to a  bounded open set (see for instance \cite[Theorem~3.2]{FFLM22}, which is proven for $N=3$ but the same arguments holds in every dimension). It is also well known that volume-constrained minimizers of the perimeter are $(\Lambda, r_0)$-minimizers of the perimeter (see for instance \cite[Example 21.3]{Maggi12}).
Thus the regularity theory  for $(\Lambda, r_0)$-minimizers (see for instance \cite[Theorems~21.8~and~28.1]{Maggi12})  applies and yields that $\pa E^S\setminus \Co$ is  a smooth constant mean curvature manifold up to a closed singular set $\Sigma_{sing}\subset \pa\Om\setminus \Co$ of Hausdorff dimension less than or equal to $N-8$. But note now that if $\Sigma_{sing}$ is nonempty, then by  the rotational symmetry of $E^S$ in the $w$-plane it follows immediately that its Hausdorff dimension is at least $N-2$, which is impossible. 
Thus we have shown that $\pa\Om\setminus \Co$ is smooth.
Moreover, arguing as in the proof of Proposition~\ref{th:ABP}, see \eqref{concav}, we deduce that $\Om$ is connected. 

Let $\Om_n$ and $\Co_n$ be the two approximating sequences provided by Lemma~\ref{appro}. Note that we may enforce that  the approximating sets $\Om_n$ additionally satisfy property (v) of  the lemma.
Let $u_n$ be  the  variational solution of 
\[
\begin{cases}
&\Delta u_n = c_n \quad \text{in } \, \Omega_n \\
&\pa_\nu u_n = 1 \quad \text{on } \, \Sigma_n\\ 
&\pa_\nu u_n = -\lambda  \quad \text{on } \, \Gamma_n, 
\end{cases}
\]
where $\Sigma_n=\pa\Om_n\setminus \Co_n$, $\Gamma_n:=\pa\Om_n\cap\Co_n$ and 
\[
c_n = \frac{\H^{N-1}(\Sigma_n) - \lambda \H^{N-1}(\Gamma_n)}{|\Omega_n|}=\frac{J_{\lambda,\Co_n}(\Om_n)}{|\Omega_n|}.
\]
By Theorem~\ref{th:lambdaABP} the convex sets $\Co_n$ satisfy the $\lambda$-ABP property, therefore as in the proof of Proposition~\ref{th:ABP}, see \eqref{eq:cabre}, we have
\beq\label{eq:cabrebis}
\begin{split}
|\Om_n|\leq |B^{\lambda}|&\leq |\mathcal A_{u_n}\cap B_1|\leq |\nabla u_n(\Om_n^+)|\leq \int_{\Om^+_n}\det\nabla^2u_n\, dx\\
&\leq \int_{\Om^+_n}\frac{(\Delta u_n)^N}{N^N}\, dx\leq \frac{(J_{\lambda,\Co_n}(\Om_n))^N}{N^N|\Om_n|^N }|\Om_n^+|\leq \frac{(J_{\lambda,\Co_n}(\Om_n))^N}{N^N|\Om_n|^N }|\Om_n|,
\end{split}
\eeq
where 
\[
\Om_n^+:=\big\{x\in \Om_n: J_{\overline \Om_n}u_n(x)\neq\emptyset\text{ and }\nabla u_n(x)\in B_1\big\}.
\]
Note that by properties (ii), (iii) and (iv) of Lemma~\ref{appro} we have that $|\Om_n|\to |\Om|=|B^\lambda|$ and
\beq\label{qeq}
J_{\lambda,\Co_n}(\Om_n)\to J_{\lambda,\Co}(\Om)=J_{\lambda,\Co}(B^\lambda)=N|B^\lambda|.
\eeq
and thus
\beq\label{qeq2}
c_n\to c:= \frac{J_{\lambda,\Co}(B^\lambda)}{|B^\lambda|}=N.
\eeq
Let us observe that up to adding  constants, we may assume that each  $u_n$ vanishes at some point $x_n$ of $\Om_n^+$.
Thus,  setting $\xi_n=\nabla u_n(x_n)\in B_1$, we have
\[
u_n(y)\geq u_n(x_n)+\xi\cdot (y-x_n)=\xi\cdot (y-x_n)\geq -\mathrm{diam\,}(\Om_n)\quad \text{for all }y\in \Om_n
\]
and thus the $u_n$'s are uniformly bounded below. Since they solve the equation $\Delta u_n=c_n$, with $c_n$ in turn uniformly bounded, it follows from a standard Harnack inequality argument that 
\[
\sup_n \|u_n\|_{\Om'}<+\infty \quad\text{for all } \Om'\subset\!\subset \Om.
\]
In turn, recalling also \eqref{qeq2} and by standard elliptic regularity, we may assume that there exists $u\in C^\infty(\Om)$ such that  up to extracting a (non relabelled) subsequence
$$
\Delta u= N\, \text{ in }\Om \quad \text{ and }\quad u_n\to u \in C^\infty(\overline \Om')\, \text{ for all } \Om'\subset\!\subset \Om.
$$
Note now that by \eqref{qeq}, the inequalities in \eqref{eq:cabrebis} become equalities in the limit for $u$. In particular, $|\Om_n^+|\to |\Om|=|B^\lambda|$ and since $|\Om_n\Delta \Om|\to 0$, we have (up to a non relabelled subsequence)  
\[
\Chi{\Om_n^+}\to \Chi{\Om}\text{ a.e. }.
\]
 We may now argue  as in the proof of Proposition~\ref{th:ABP} to deduce that $|\nabla u(x)|\leq 1$ and 
 $\nabla^2u=I$ in $ \Om$ and thus, by the connectedness of $\Om$, there exist $x_0\in \R^N$ and $b\in \R$ such that
$$
 u(x)=\frac12|x-x_0|^2+b\quad\text{for all }x\in \Om.
$$

 We now study the boundary conditions satisfied by $u$. To this aim, let $B_r(x)\subset\!\subset \R^N\setminus \Co$. By (v) of Lemma~\ref{appro} we have that $\pa\Om_n\cap B_r(x)$  converge in $C^\infty$ to $\pa \Om\cap B_r(x)$. Since
 \beq\label{wform}
 \int_{B_r(x)\cap \Om_n }\nabla u_n\cdot \nabla\varphi=-c_n\int_{B_r(x)\cap \Om_n}\varphi\, dx+\int_{\pa \Om_n\cap B_r(x)}\varphi\, d\H^{N-1}\quad\text{for all }\varphi\in H^1_0(B_r(x)),
 \eeq
by a standard Caccioppoli Inequality argument and exploiting that   Trace Theorem holds on $\pa \Om_n\cap B_r(x)$ with uniformly bounded constants, we deduce  $\sup_n\|u_n\|_{H^1(\Om_n\cap B_r(x))}<+\infty$ and thus, in particular,  $u_n\wto u$ weakly in $H^1(\Om\cap B_r(x))$. Therefore, we can pass to the limit in \eqref{wform} to get
\[
\int_{B_r(x)\cap \Om }\nabla u\cdot \nabla\varphi=-N\int_{B_r(x)\cap \Om}\varphi\, dx+\int_{\pa \Om\cap B_r(x)}\varphi\, d\H^{N-1}\quad\text{for all }\varphi\in C^\infty_c(B_r(x)),
\]
 which yields $\pa_\nu u=1$  on $\pa\Om\cap B_r(x)$ and thus on $\pa\Om\setminus \Co$ by the arbitrariness of $B_r(x)$. 
  
 Note now that as $\Co_n\to \Co$ in the Kuratowski sense we have that the boundaries $\pa \Co_n$ are locally equi-Lipschitz.
Thus for any ball $B_r(x)$ such that $\overline B_r(x)\cap \pa \Om=(\Gamma\setminus \gamma)\cap \overline B_r(x)$, we have eventually  $\overline B_r(x)\cap \pa \Om_n=(\Gamma_n\setminus \gamma_n)\cap \overline B_r(x)$ and 
 we may extend each  $u_n$ to a function $\tilde u_n\in H^1(B_r)$ with equibounded $H^1$-norms. Therefore, up to a non ralebelled subsequence, we may assume $\tilde u_n\wto \tilde u$ in $H^1(B_r(x))$, with $\tilde u=u$ in $\Om\cap B_r(x)$. We may now argue similarly as before, to deduce  that $\pa_\nu u=-\lambda$  a.e. in  $\Gamma\setminus \gamma$.   
We may now  argue as in  the final part of the  proof of Proposition~\ref{th:ABP} to conclude that $\Om$ is a spherical cap isometric to $B^\lambda$ sitting on a facet of $\Co$.  

To conclude the theorem, we observe that the function $m$ defined in $\eqref{mslice}$ is smooth in $\R^2\setminus\mathcal C$ (as it coincides with that of the spherical cap). Hence, by \cite[Theorem~1.2 and Proposition~3.5]{BCF13} we may conclude that $E$ is equivalent to a translate of  $E^S=\Om$. This concludes the proof of the theorem.  
\end{proof}
\begin{remark}\label{empty} We observe here that Theorem~\ref{thm1} extends to the case convex cylinders with empty interior of the form $\Co=I\times\R^{N-2}$, where $I\subset\R$ is any closed interval. In this case the capillary energy must be defined as follows
\beq\label{empty1}
\J(E):=P(E; \R^N \setminus \Co) - \lambda\int_{\Co}\big(\mathrm{Tr}^+(\Chi{E})+\mathrm{Tr}^-(\Chi{E})\big)\,d\H^{N-1},
\eeq
where $\mathrm{Tr}^\pm(\Chi{E})$ denote the traces of the characteristic function $\Chi{E}$ on both sides of $\Co$, see for instance \cite[Theorem 3.77]{AmbrosioFuscoPallara00}. Indeed, this follows  easily after observing  that the right-hand side of \eqref{empty1} is the limit of $J_{\lambda,\Co_n}(E\setminus\Co_n)$, where $\Co_n$ denotes the closed $\frac{1}{n}$-neighborhood of $\Co$.
\end{remark}

\section{Appendix}

\begin{proof}[\textbf{Proof of Lemma~\ref{appro}}] Assume first that $E=\Om$, where $\Om$ is a bounded open set of finite perimeter. Let $B_R$ be a ball such that $\Om\subset\!\subset B_R$ and assume without loss of generality that $\C$ contains the origin of $\R^2$.
\par
Given $\sigma>0$  we begin by constructing a sequence of smooth convex sets $\C_{\sigma}^k\subset\R^2$,  with $\C\subset \C_{\sigma}^k$,  converging to $(1+\sigma)\C$ in the Kuratowski sense as $k\to\infty$ and such that  $(1+\sigma)\C\cap D_R\subset\C_{\sigma}^k\cap D_R$, where $D_R$ is the two-dimensional disk with radius $R$. Up to slightly dilating $\C_{\sigma}^k$ if needed, we may always assume that, setting $\Co_{\sigma}^k=\C_{\sigma}^k\times\R^{N-2}$, we have
\beq\label{appro1}
\H^{N-1}(\pa\Om\cap\pa\Co_{\sigma}^k)=0 \qquad \text{for all } \,  k,\sigma.
\eeq
We consider the  signed distance function ${\rm sd}_{\Co_{\sigma}^k}(x)$ from $\pa\Co_{\sigma}^k$,  which is a $C^\infty$ function in $O_{\sigma}^k=\{x:\,{\rm sd}_{\Co_{\sigma}^k}(x)>-\eta_{\sigma}^k\}$ for some $\eta_{\sigma}^k>0$. Consider the smooth convex sets $\Co_{\sigma,s }^k:=\{x:\,{\rm sd}_{\Co_{\sigma}^k(x)}\leq s\}$ for $s>-\eta_{\sigma}^k$. 

To approximate $\Om$ we first extend $\Chi{\Om}_{\big|_{\R^N\setminus\Co}}$ to a function $u\in BV(\R^N)$, with compact support, such that $|Du|(\pa\Co)=0$, $0\leq u\leq1$, see \cite[Proposition~3.21]{AmbrosioFuscoPallara00}. Note that for all $t\in(0,1)$, $\{u>t\}\setminus\Co=\Om$. For any $\e>0, t\in(0,1)$ we set $U_{\e,t}=\{x:\,u_\e(x)>t\}$, where $u_\e=\varrho_\e*u$, for a standard mollifier $\varrho_\e$. 
Note that for a.e. $t\in(0,1)$ there exists a sequence $\e_n$ converging to zero such that  and 
\begin{equation}\label{appro2}
\begin{split}
& \lim_{n\to\infty}|U_{\e_n,t}\triangle\{u>t\}|=0, \quad \lim_{n\to\infty}P(U_{\e_n,t})=P(\{u>t\}), \\
&\pa U_{\e_n,t}\subset\Big\{x:\,{\rm dist}(x,\pa\{u>t\})<\frac{1}{n}\Big\},
\end{split}
\end{equation}
see \cite[Theorem 3.42]{AmbrosioFuscoPallara00}. 
\par
Consider now the $C^\infty$ map  $x\mapsto ({\rm sd}_{\Co_{\sigma}^k}(x),u_\e(x))$ defined  for all $x\in O_{\sigma}^k$. By Sard's theorem we have that 
$$
{\rm rank}\bigg(\!\!
\begin{array}{c}
\nabla  {\rm sd}_{\Co_{\sigma}^k}(x) \\
\nabla u_{\e_n}(x)
\end{array}
\!\!
\bigg)=2 \quad\text{on $\big\{x:\,{\rm sd}_{\Co_{\sigma}^k}(x)=s,\,\,u_{\e_n}(x)=t\big\}$ for a.e. $(s,t)\in(0,\infty)\times (0,1)$.}
$$
Hence, we may fix from now on $t\in (0,1)$ satisfying \eqref{appro2} and such that the or a.e. $s>0$ the above rank condition holds for all $n$. Therefore for a.e. $s>0$ the open set $\Om_{\sigma,\e_n,s}^k=U_{\e_n,t}\setminus\Co_{\sigma,s}^k$ is a Lipschitz domain such that $\pa\Om_{\sigma,\e_n,s}^k\setminus \Co_{\sigma,s}^k$ is a $C^\infty$ manifold with boundary. Note that for any $\sigma$ and $k$ we have that for a.e. $s$, $\H^{N-1}(\pa\Om\cap\pa\Co_{\sigma,s}^k)=0$. Therefore for all such $s$, we have
\beq\label{appro2.5}
\begin{split}
\lim_{n\to\infty}P(\Om_{\sigma,\e_n,s}^k;\R^N\setminus\Co_{\sigma,s}^k)
&=\lim_{n\to\infty}P(U_{\e_n,t};\R^N\setminus\Co_{\sigma,s}^k)\\
&=P(\Om;\R^N\setminus\Co_{\sigma,s}^k)=P(\Om\setminus\Co_{\sigma,s}^k;\R^N\setminus\Co_{\sigma,s}^k).
\end{split}
\eeq
From the above convergence and the continuity of the trace Theorem for $BV$ functions, see \cite[Theorem~3.88]{AmbrosioFuscoPallara00} we have that
\beq\label{appro2.6}
\lim_{n\to\infty}\H^{N-1}(\pa\Om_{\sigma,\e_n,s}^k\cap\pa\Co_{\sigma,s}^k))=\H^{N-1}(\pa^*(\Om\setminus\Co_{\sigma,s}^k)\cap\pa\Co_{\sigma,s}^k)=\H^{N-1}(\Om\cap\pa\Co_{\sigma,s}^k),
\eeq
where the last equality follows from the fact that $\H^{N-1}(\pa\Om\cap\pa\Co_{\sigma,s}^k)=0$. Observe now that, since  $\Co_{\sigma,s}^k$ converge to $\Co_{\sigma}^k$ in the Kuratowski sense, as $s\to0$, we have in particular that $\H^{N-1}\res\pa\Co_{\sigma,s}^k\wstar\H^{N-1}\res\pa\Co_{\sigma}^k$, see for instance \cite[Remark~2.2]{FFLM22}. Therefore, thanks to \eqref{appro1} we conclude that $\H^{N-1}(\Om\cap\pa\Co_{\sigma,s}^k)\to\H^{N-1}(\Om\cap\pa\Co_{\sigma}^k)=\H^{N-1}(\pa^*(\Om\setminus\Co_{\sigma}^k)\cap\pa\Co_{\sigma}^k)$. Thus we have
\beq\label{appro3}
\begin{split}
&\lim_{s\to0}P(\Om\setminus\Co_{\sigma,s}^k;\R^N\setminus\Co_{\sigma,s}^k)=P(\Om\setminus\Co_{\sigma}^k;\R^N\setminus\Co_{\sigma}^k)\\
&\lim_{s\to0}\H^{N-1}(\pa^*(\Om\setminus\Co_{\sigma,s}^k)\cap\pa\Co_{\sigma,s}^k)=\H^{N-1}(\pa^*(\Om\setminus\Co_{\sigma}^k)\cap\pa\Co_{\sigma}^k).
\end{split}
\eeq
 By a similar argument, if $\sigma>0$ is such that $\H^{N-1}(\pa\Om\cap\pa(1+\sigma)\Co)=0$, setting $\Co_\sigma=(1+\sigma)\Co$ we have
\beq\label{appro4}
\begin{split}
&\lim_{k\to\infty}P(\Om\setminus\Co_{\sigma}^k;\R^N\setminus\Co_{\sigma}^k)=P(\Om\setminus\Co_\sigma;\R^N\setminus\Co_\sigma),\\
&\lim_{k\to\infty}\H^{N-1}(\pa^*(\Om\setminus\Co_{\sigma}^k)\cap\pa\Co_{\sigma}^k)=\H^{N-1}(\pa^*(\Om\setminus\Co_\sigma)\cap\pa\Co_\sigma).
\end{split}
\eeq
Finally, we note that by monotone convergence
\beq\label{appro5}
\lim_{\sigma\to0}P(\Om\setminus\Co_\sigma;\R^N\setminus\Co_\sigma)=\lim_{\sigma\to0}P(\Om;\R^N\setminus\Co_\sigma)=P(\Om;\R^N\setminus\Co).
\eeq
By scaling, this is equivalent to say that 
$$
\lim_{\sigma\to0}P\big(\big((1+\sigma)^{-1}\Om\big)\setminus\Co;\R^N\setminus\Co\big)=P(\Om;\R^N\setminus\Co).
$$
Therefore, the trace theorem again implies that 
\[
\begin{split}
&\lim_{\sigma \to0}\H^{N-1}(\pa^*(\Om\setminus\Co_\sigma)\cap\pa\Co_\sigma)\\
&=\lim_{\sigma\to0}(1+\sigma)^{N-1}\H^{N-1}\big(\pa^*\big((1+\sigma)^{-1}\Om\big)\setminus\Co\big)\cap\pa\Co)
=\H^{N-1}(\pa^*\Om\cap\pa\Co).
\end{split}
\]
From this equality, together with \eqref{appro2.5}-\eqref{appro5} we conclude, by a diagonal argument,  that there exist sequences $s_n\to0^+$, $k_n\to\infty$ and $\sigma_n\to0^+$ such that, setting $\Om_n=\Om_{\sigma_n,\e_n,s_n}^{k_n}$, $\Co_n=\Co^{k_n}_{\sigma_n,s_n}$, (i)--(iv) hold. 

If $E$ is now a general set of finite perimeter, the conclusion  follows through a further diagonal procedure by approximating $E$ with a sequence of bounded open sets $\{\Om_n\}$ contained in $\R^N\setminus \Co$ in such way that $|\Om_n\Delta E|\to 0$ and $P(\Om_n; \R^N\setminus \Co)\to P(E; \R^N\setminus \Co)$, which implies $\H^{N-1}(\pa\Om_n\cap\pa\Co_n)\to\H^{N-1}(\pa^*E\cap\pa\Co)$ by the continuity of the trace Theorem. 

Let us now consider the case where $E$ coincides with an open set of finite perimeter $\Om\subset \R^N\setminus \Co$ such that  $\pa\Om\setminus \Co$ is smooth. In this case we can simplify the approximation, by considering the signed distance function $\mathrm{sd}_\Om$ to  the boundary of $\Om$. By the smoothness  of $\pa\Om\setminus \Co$ and the properties of signed distance functions, we have that for all $\sigma>0$ there exists $\e(\sigma)>0$ such that $\mathrm{sd}_\Om$ is smooth in $(\pa\Om)_{\e(\sigma)}\setminus (1+\sigma)\Co$.
Here $(\pa\Om)_{\e(\sigma)}$ denotes the $\e(\sigma)$-neighborhood of $\pa \Om$.
 The idea is now 
to proceed as before, with $u_\e$ replaced by $\mathrm{sd}_\Om$,  the set $U_{\e, t}$ replaced by 
$(\Om)_\e:=\{x: \mathrm{sd}_\Om\leq \e\}$, and with  $\Om_{\sigma,\e,s}^k:=(\Om)_\e\setminus\Co_{\sigma,s}^k$. Note that again by Sard's theorem we have
$$
{\rm rank}\bigg(\!\!
\begin{array}{c}
\nabla  {\rm sd}_{\Co_{\sigma}^k}(x) \\
\nabla \mathrm{sd}_\Om(x)
\end{array}
\!\!
\bigg)=2 \quad\text{on $\big\{x:\,{\rm sd}_{\Co_{\sigma}^k}(x)=s,\,\,\mathrm{sd}_\Om(x)=\e\big\}$ for a.e. $(s,\e)\in(0,\infty)\times (0,\e(\sigma))$.}
$$  
Hence, for all such $(s, \e)$ the set  $\Om_{\sigma,\e,s}^k$ is a Lipschitz domain such that $\pa\Om_{\sigma,\e,s}^k\setminus \Co_{\sigma,s}^k$ is a $C^\infty$ manifold with boundary. Moreover, for a.e. $s>0$ and for all $k$ we can always find a sequence $\e_n\to 0$  such that $(s, \e_n)$ satisfies the above rank condition for all $n$. Along  this subsequence,  we easily get  that $\big|\Om_{\sigma,\e_n,s}^k\Delta (\Om\setminus\Co_{\sigma,s}^k)\big |\to 0$ 
and that
\eqref{appro2.5} and \eqref{appro2.6} hold.  We may now proceed as before to reach the conclusion.
\end{proof}

\section*{Aknowledgments} 
\noindent N.~F. and M.~M. have been supported by PRIN 2022 Project “Geometric Evolution Problems and Shape Optimization (GEPSO)”, PNRR Italia Domani, financed by European Union via the Program NextGenerationEU, CUP D53D23005820006.  N.~F. and M.~M.  are  members of the Gruppo Nazionale per l’Analisi
Matematica, la Probabilit\`a e le loro Applicazioni (GNAMPA), which is part of the Istituto
Nazionale di Alta Matematica (INdAM). V.~J.~was supported by the Academy of Finland grant 314227.

\bibliographystyle{siam}
\bibliography{isoconvex}

\begin{thebibliography}{10}

\bibitem{AmbrosioFuscoPallara00}
{\sc L.~Ambrosio, N.~Fusco, and D.~Pallara}, {\em {Functions of bounded
  variation and free discontinuity problems}}, {Oxford Mathematical
  Monographs}, The Clarendon Press, Oxford University Press, New York, 2000.

\bibitem{BCF13}
{\sc M.~Barchiesi, F.~Cagnetti, and N.~Fusco}, {\em Stability of the {S}teiner
  symmetrization of convex sets}, J. Eur. Math. Soc. (JEMS), 15 (2013),
  pp.~1245--1278.

\bibitem{cabre2000}
{\sc X.~Cabr\'{e}}, {\em Partial differential equations, geometry and
  stochastic control}, Butl. Soc. Catalana Mat., 15 (2000), pp.~7--27.

\bibitem{cabre2008}
\leavevmode\vrule height 2pt depth -1.6pt width 23pt, {\em Elliptic {PDE}'s in
  probability and geometry: symmetry and regularity of solutions}, Discrete
  Contin. Dyn. Syst., 20 (2008), pp.~425--457.

\bibitem{CRS2016}
{\sc X.~Cabr\'{e}, X.~Ros-Oton, and J.~Serra}, {\em Sharp isoperimetric
  inequalities via the {ABP} method}, J. Eur. Math. Soc. (JEMS), 18 (2016),
  pp.~2971--2998.

\bibitem{ChGhoRi06}
{\sc J.~Choe, M.~Ghomi, and M.~Ritor\'{e}}, {\em Total positive curvature of
  hypersurfaces with convex boundary}, J. Differential Geom., 72 (2006),
  pp.~129--147.

\bibitem{ChGhoRi07}
\leavevmode\vrule height 2pt depth -1.6pt width 23pt, {\em The relative
  isoperimetric inequality outside convex domains in {${\bf R}^n$}}, Calc. Var.
  Partial Differential Equations, 29 (2007), pp.~421--429.

\bibitem{CGPRS}
{\sc E.~Cinti, F.~Glaudo, A.~Pratelli, X.~Ros-Oton, and J.~Serra}, {\em Sharp
  quantitative stability for isoperimetric inequalities with homogeneous
  weights}, Trans. Amer. Math. Soc., 375 (2022), pp.~1509--1550.

\bibitem{De-PhilippisMaggi15}
{\sc G.~De~Philippis and F.~Maggi}, {\em Regularity of free boundaries in
  anisotropic capillarity problems and the validity of {Y}oung's law}, Arch.
  Ration. Mech. Anal., 216 (2015), pp.~473--568.

\bibitem{dupuisishii}
{\sc P.~Dupuis and H.~Ishii}, {\em On oblique derivative problems for fully
  nonlinear second-order elliptic {PDE}s on domains with corners}, Hokkaido
  Math. J., 20 (1991), pp.~135--164.

\bibitem{Finnbook}
{\sc R.~Finn}, {\em Equilibrium capillary surfaces}, vol.~284 of Grundlehren
  der mathematischen Wissenschaften [Fundamental Principles of Mathematical
  Sciences], Springer-Verlag, New York, 1986.

\bibitem{FFLM22}
{\sc I.~Fonseca, N.~Fusco, G.~Leoni, and M.~Morini}, {\em Global and local
  energy minimizers for a nanowire growth model}, Ann. Inst. H. Poincar\'{e} C
  Anal. Non Lin\'{e}aire, 40 (2023), pp.~919--957.

\bibitem{FMMN}
{\sc N.~Fusco, F.~Maggi, M.~Morini, and M.~Novack}, {\em Rigidity and large
  volume residues in exterior isoperimetry for convex sets}, 2023.
\newblock ArXiv preprint 2310.13569.

\bibitem{FM2023}
{\sc N.~Fusco and M.~Morini}, {\em Total positive curvature and the equality
  case in the relative isoperimetric inequality outside convex domains}, Calc.
  Var. Partial Differential Equations, 62 (2023), pp.~Paper No. 102, 32.

\bibitem{JWXZ}
{\sc X.~Jia, G.~Wang, C.~Xia, and X.~Zhang}, {\em Heintze-karcher inequality
  and capillary hypersurfaces in a wedge}, 2023.
\newblock ArXiv preprint 2209.13839.

\bibitem{krummel}
{\sc B.~Krummel}, {\em Higher codimension relative isoperimetric inequality
  outside a convex set}, 2017.
\newblock ArXiv preprint 1710.04821.

\bibitem{LiuWangWeng}
{\sc L.~Liu, G.~Wang, and L.~Weng}, {\em The relative isoperimetric inequality
  for minimal submanifolds with free boundary in the {E}uclidean space}, J.
  Funct. Anal., 285 (2023), pp.~Paper No. 109945, 22.

\bibitem{lopez}
{\sc R.~L\'{o}pez}, {\em Capillary surfaces with free boundary in a wedge},
  Adv. Math., 262 (2014), pp.~476--483.

\bibitem{Maggi12}
{\sc F.~Maggi}, {\em {Sets of finite perimeter and geometric variational
  problems: an introduction to Geometric Measure Theory}}, vol.~135 of
  {Cambridge Studies in Advanced Mathematics}, Cambridge University Press,
  2012.

\bibitem{gasparetto}
{\sc L.~D. Masi, N.~Edelen, C.~Gasparetto, and C.~Li}, {\em Regularity of
  minimal surfaces with capillary boundary conditions}, 2024.
\newblock ArXiv preprint 2405.20796.

\bibitem{Nittka}
{\sc R.~Nittka}, {\em Regularity of solutions of linear second order elliptic
  and parabolic boundary value problems on {L}ipschitz domains}, J.
  Differential Equations, 251 (2011), pp.~860--880.

\bibitem{paspoz}
{\sc G.~Pascale and M.~Pozzetta}, {\em Quantitative isoperimetric inequalities
  for classical capillarity problems}, 2024.
\newblock ArXiv preprint 2402.04675.

\bibitem{Taylor77}
{\sc J.~E. Taylor}, {\em Boundary regularity for solutions to various
  capillarity and free boundary problems}, Comm. Partial Differential
  Equations, 2 (1977), pp.~323--357.

\end{thebibliography}

\end{document}